\documentclass[final,a4paper,11pt]{article}
\usepackage{mathptmx}
\usepackage{amssymb,amsmath, amsfonts,amsthm}
\usepackage{a4wide}
\usepackage{color}
\usepackage{cases}

\newcommand{\email}[1]{{\tt #1}}
\newcommand{\R}{\mathbb{R}}
\newcommand{\N}{\mathbb{N}}
\newcommand{\norm}[1]{\|#1\|}

\DeclareMathOperator{\ddist}{dist}
\newcommand{\dist}[1]{\ddist(#1)}

\newcommand{\mv}{\mid\;}

\newcommand{\B}{{\cal B}}

\newcommand{\I}{{\cal I}}

\newcommand{\G}{{\cal G}}
\newcommand{\K}{{\cal K}}
\newcommand{\M}{{\cal M}}

\newcommand{\U}{{\cal U}}
\newcommand{\Sp}{{\cal S}}
\newcommand{\F}{{\cal F}}
\newcommand{\Z}{{\cal Z}}

\newcommand{\setto}[1]{\mathop{\rightarrow}\limits^#1}
\newcommand{\longsetto}[1]{\mathop{\longrightarrow}\limits^{#1}}
\newcommand{\skalp}[1]{\langle #1\rangle}

\newcommand{\xb}{\bar x}
\newcommand{\yb}{\bar y}

\newcommand{\AT}[2]{{\textstyle{#1\atop#2}}}
\newcommand{\xba}{{\bar x^\ast}}

\newcommand{\oo}{o}
\newcommand{\OO}{{\cal O}}
\newcommand{\argmin}{\mathop{\rm arg\,min}}

\newcommand{\cl}{{\rm cl\,}}

\newcommand{\co}{{\rm conv\,}}
\newcommand{\gph}{\mathrm{gph}\,}
\newcommand{\dom}{\mathrm{dom}\,}
\newcommand{\epi}{\mathrm{epi}\,}
\newcommand{\tto}{\rightrightarrows}
\newcommand{\Limsup}{\mathop{{\rm Lim}\,{\rm sup}}}
\newcommand{\Liminf}{\mathop{{\rm Lim}\,{\rm inf}}}

\newcommand{\myvec}[1]{\left(\begin{array}{c}#1\end{array}\right)}
\newcommand{\SCD}{SCD\ }
\newcommand{\subreg}{{\rm subreg\,}}
\newcommand{\lsubreg}{{\rm l\mbox{-}subreg\,}}
\newcommand{\scdreg}{{\rm scd\,reg\;}}
\newcommand{\reg}{{\rm reg\,}}

\DeclareSymbolFont{largesymbolsA}{U}{txexa}{m}{n}
\DeclareMathSymbol{\varprod}{\mathop}{largesymbolsA}{16}

\newcommand{\ssstar}{semismooth$^{*}$ }

\newcommand{\rge}{{\rm rge\;}}

\newlength{\myAlgBox}
\newtheorem{theorem}{Theorem}[section]
\newtheorem{proposition}[theorem]{Proposition}
\newtheorem{remark}[theorem]{Remark}
\newtheorem{lemma}[theorem]{Lemma}
\newtheorem{corollary}[theorem]{Corollary}
\newtheorem{definition}[theorem]{Definition}
\newtheorem{example}[theorem]{Example}
\newtheorem{algorithm}{Algorithm}

\title{On (local) analysis of multifunctions via subspaces contained in graphs of generalized derivatives}
\author{Helmut Gfrerer\thanks{Institute of Computational Mathematics, Johannes Kepler University
Linz, A-4040 Linz, Austria; \email{helmut.gfrerer@jku.at}}
 \and   Ji\v{r}\'{i} V. Outrata\thanks{Institute of Information Theory and Automation, Czech Academy of
 Sciences, 18208 Prague, Czech Republic, and Centre for
              Informatics and Applied Optimization, Federation University of Australia, POB 663,
              Ballarat,  Vic 3350, Australia,  \email{outrata@utia.cas.cz}}
}
\date{}

\begin{document}
\maketitle
{\footnotesize
{\bf Abstract.} The paper deals with a comprehensive theory  of mappings, whose local behavior can be described by means of linear subspaces, contained in the graphs of two (primal and dual) generalized derivatives. This class of mappings includes
the graphically Lipschitzian mappings and thus a number of multifunctions, frequently arising in optimization and
equilibrium problems. The developed theory makes use of new generalized derivatives, provides us with some calculus
rules and reveals a number of interesting connections. In particular, it enables us to construct a  modification of the semismooth* Newton method with improved convergence properties and to derive a generalization of Clarke's Inverse Function Theorem to multifunctions together with new efficient characterizations of strong metric (sub)regularity and tilt stability.

{\bf Key words.} generalized derivatives, second-order theory, strong metric (sub)regularity, semismoothness${}^*$.

{\bf AMS Subject classification.} 65K10, 65K15, 90C33.
}

\section{Introduction}
When implementing the \ssstar Newton method \cite{GfrOut21} for solving an inclusion of the form
\[0\in F(x)\]
with some set-valued mapping $F:\R^n\tto\R^n$, we observed  that it is advantageous to work with linear subspaces $L\subseteq \R^n\times\R^n$ having dimension $n$ and  contained in the graph of the limiting coderivative, i.e.,
\[L\subseteq \gph D^*F(x,y)\]
at points $(x,y)\in\gph F$.

However, this paper goes far beyond the analysis of the above issue  and presents a comprehensive study of a class of mappings, whose local behavior can be described by appropriately constructed linear subspaces. This class turns out to be rather broad and the developed theory  helps us both to suggest an efficient modification of the \ssstar Newton method as well as to derive a number of new results concerning strong metric subregularity, strong metric regularity and tilt stability. More precisely, for the mentioned mapping $F$, primal and dual generalized derivatives are introduced, whose elements are subspaces of dimension $n$. In order to define these derivatives, we consider points in the graph of the mapping where the tangent cone amounts to a subspace and then perform an  outer limiting operation in a certain compact metric space.

Our construction is motivated by the definition of the  B-subdifferential (Bouligand-subdifferential) for  single-valued mappings,  whose elements are given as limit of Jacobians at points where the mapping is Fr\'echet differentiable. Note that the tangent cone to the graph of a map is a subspace whenever the mapping is differentiable at the point under consideration. Instead of computing limits of matrices, we consider the limit of subspaces given by the graph of the linear mappings induced by the matrices. When the mapping is Lipschitzian, then we obtain a one-to-one correspondence between the new primal generalized derivative and the B-subdifferential. However, for non-Lipschitzian single-valued mappings there will be a difference because we are considering limits of subspaces in a compact metric space whereas the underlying matrices can be unbounded.

There are some relations between our generalized derivatives and existing ones. The  dual derivative consists of subspaces  contained in the limiting coderivative and the elements of the primal derivative are subspaces contained in the so-called outer limiting graphical derivative. To the best of our knowledge, the latter has not yet been considered in the literature and is contained in the so-called strict graphical derivative.

Our theory is not applicable to arbitrary mappings. However, as already mentioned, the class of mappings which are suited for our approach, is rather broad and important for applications. In particular, every mapping which is graphically Lipschitzian, i.e., its graph coincides under some change of coordinates with the graph of a locally Lipschitzian mapping, belongs to this class. Graphically Lipschitzian mappings have been already considered by Rockafellar \cite{Ro85}. E.g., locally maximally hypomonotone mappings like the subdifferential mapping of prox-regular and subdifferentially continuous functions possess this property \cite{PolRo96}. Thus, our approach is particularly suitable for second-order theory and we will establish a strong relationship with the so-called quadratic bundle introduced in the recent paper \cite{Ro21}.
Note that in  \cite{Ro85} also a limit of tangent spaces  has been considered. However, in \cite{Ro85} an inner limit with respect to the usual set-convergence has been used yielding a different sort of results.

Within the framework of the new theory one can introduce a new regularity notion leading to an adaptation of the \ssstar Newton method. This notion is weaker than metric regularity and enables us to streamline the algorithm and to relax the assumptions, ensuring its locally superlinear convergence. Under the respective regularity condition it is also possible to show that a \ssstar mapping  is strongly metrically subregular, not only at the reference point itself but also on a neighborhood of it.  It seems that this somewhat extended property of strong metric subregularity {\em around} the reference point has not been considered yet. As a byproduct, we present a characterization of this property by means of the outer limiting graphical derivative.

Finally we turn our attention to the property of strong metric regularity. Since strongly regular mappings are graphically Lipschitzian  by the definition, the preceding theory enables us to reveal some interesting new connections. In particular, one obtains a generalization of Clarke's Inverse Function Theorem to set-valued mappings and, when applied to specific problem classes, these results lead to new characterizations of strong metric regularity for locally maximally monotone operators and to a new characterization of tilt-stability. Compared with existing characterizations, the new ones have the advantage, that not the whole strict graphical derivative or limiting coderivative must be checked (as, e.g., in \cite{KlKum02,DoFra13,PolRo98}) but only a condition on the subspaces contained in its graph. In this way the arsenal of available criteria for strong metric regularity (cf. \cite{Ro80, DoRo96, KlKum02, Gow04, DoFra13, DoRo14}) and tilt stability (cf.\cite{PolRo98, DruLew13, MoNg15}) is enriched.

The plan of the paper is as follows. After the Preliminaries, devoted to relevant notions from variational analysis, in Section 3 we introduce and analyze the crucial class of SCD (subspace containing derivative) mappings. In their analysis we make use of the mentioned generalized derivatives, for which some basic calculus rules are  developed and exact formulas in case of graphically Lipschitzian mappings are provided. In Section 4 we introduce the notion of SCD-regularity, which plays a central role in the subsequent sections. Section 5 deals with the adaptation of the \ssstar Newton method to SCD mappings. The property of strong metric subregularity around a point is characterized in Section 6. Finally, in Section 7 we present a generalization of Clarke's Inverse Function Theorem and new characterizations of strong metric regularity and tilt stability for various classes of SCD mappings.
\if
Section 4 deals with the adaptation of the \ssstar Newton method to
SCD mappings. The convergence analysis leads in this context to the notion of SCD regularity, whose relationship with the strong metric subregularity around a point is established in Section 5. Finally, in Section 6 we present a generalization of Clarke's Inverse Function Theorem and new characterizations of strong metric regularity and tilt stability for various classes of SCD mappings.
\fi

The following notation is employed. Given a linear subspace $L\subseteq \R^n$, $ L^\perp$ denotes its orthogonal
complement and, for a closed cone $K$ with vertex at the origin, $K^\circ$ signifies its (negative) polar.
Further, given a multifunction $F$, $\gph F:=\{(x,y)\mv y\in F(x)\}$ stands for its
graph. For an element $u\in\R^n$, $\norm{u}$ denotes its Euclidean norm and  $\B_\delta(u)$ denotes the closed ball around $u$
with radius $\delta$. In a product space we use the norm $\norm{(u,v)}:=\sqrt{\norm{u}^2+\norm{v}^2}$. Given
an $m\times n$ matrix $A$, we employ the operator norm $\norm{A}$ with respect to the Euclidean norm and we denote the range of A by $\rge A$. Given a set $\Omega\subset\R^s$, we define the distance of a point $x$ to $\Omega$ by $d_\Omega(x):=\dist{x,\Omega}:=\inf\{\norm{y-x}\mv y\in\Omega\}$ and the indicator function is denoted by $\delta_\Omega$. When a mapping $F:\R^n\to\R^m$ is differentiable at $x$, we denote by $\nabla F(x)$ its Jacobian.

\section{Preliminaries}
Throughout the whole paper, we will frequently use  the following basic notions of modern
variational analysis. All the sets under consideration are supposed to be {\em locally closed} around the points in question without further mentioning.
 \begin{definition}\label{DefVarGeom}
 Let $A$  be a  set in $\mathbb{R}^{s}$ and let $\bar{x} \in A$. Then
\begin{enumerate}
 \item [(i)]The  {\em tangent (contingent, Bouligand) cone} to $A$ at $\bar{x}$ is given by
 \[T_{A}(\bar{x}):=\Limsup\limits_{t\downarrow 0} \frac{A-\bar{x}}{t}\]
   and the {\em paratingent cone} to $A$ at $\xb$ is given by
 \[T^P_A(\xb):=\Limsup\limits_{\AT{x\setto{{A}}\xb}{t\downarrow 0}} \frac{A- x}{t}\]
 \item[(ii)] The set
 \[\widehat{N}_{A}(\bar{x}):=(T_{A}(\bar{x}))^{\circ}\]
 is the {\em regular (Fr\'{e}chet) normal cone} to $A$ at $\bar{x}$, and

 \[N_{A}(\bar{x}):=\Limsup\limits_{\stackrel{A}{x \rightarrow \bar{x}}} \widehat{N}_{A}(x)\]
 is the {\em limiting (Mordukhovich) normal cone} to $A$ at $\bar{x}$. Given a direction $d
 \in\mathbb{R}^{s}$,
\[ N_{A}(\bar{x};d):= \Limsup\limits_{\stackrel{t\downarrow 0}{d^{\prime}\rightarrow
 d}}\widehat{N}_{A}(\bar{x}+ td^{\prime})\]
 is the {\em directional limiting normal cone} to $A$ at $\bar{x}$ {\em in direction} $d$.
 \end{enumerate}

\end{definition}
In this definition ''$\Limsup$'' stands for the Painlev\' e-Kuratowski {\em outer (upper) set limit}, see, e.g.,\cite{AubFra90}.
If $A$ is convex, then $\widehat{N}_{A}(\bar{x})= N_{A}(\bar{x})$ amounts to the classical normal cone in
the sense of convex analysis and we will  write $N_{A}(\bar{x})$. By the definition, the limiting normal
cone coincides with the directional limiting normal cone in direction $0$, i.e.,
$N_A(\bar{x})=N_A(\bar{x};0)$, and $N_A(\bar{x};d)=\emptyset$ whenever $d\not\in T_A(\bar{x})$.

The above listed cones enable us to describe the local behavior of set-valued maps via various
generalized derivatives. All the set-valued mappings under consideration are supposed to have {\em locally closed graph} around the points in question.

\begin{definition}\label{DefGenDeriv}
Consider a  multifunction $F:\R^n\tto\R^m$ and let $(\xb,\yb)\in \gph F$.
\begin{enumerate}
\item[(i)]
 The multifunction $DF(\xb,\yb):\R^n\tto\R^m$ given by $\gph DF(\xb,\yb)=T_{\gph F}(\xb,\yb)$ is called the {\em graphical derivative} of $F$ at $(\xb,\yb)$.
\item[(ii)]
 The multifunction $D_*F(\xb,\yb):\R^n\tto\R^m$ given by $\gph D_*F(\xb,\yb)=T^P_{\gph F}(\xb,\yb)$ is called the {\em  strict (paratingent) derivative} of $F$ at $(\xb,\yb)$.
\item[(iii)]
 The multifunction $\widehat D^\ast F(\xb,\yb ): \R^m\tto\R^n$  defined by
 \[ \gph \widehat D^\ast F(\xb,\yb )=\{(y^*,x^*)\mv (x^*,-y^*)\in \widehat N_{\gph F}(\xb,\yb)\}\]
is called the {\em regular (Fr\'echet) coderivative} of $F$ at $(\xb,\yb )$.
\item [(iv)]  The multifunction $D^\ast F(\xb,\yb ): \R^m \tto \R^n$,  defined by
 \[ \gph D^\ast F(\xb,\yb )=\{(y^*,x^*)\mv (x^*,-y^*)\in N_{\gph F}(\xb,\yb)\}\]
is called the {\em limiting (Mordukhovich) coderivative} of $F$ at $(\xb,\yb )$.
\item [(v)]
 Given a pair of directions $(u,v) \in \R^n \times \R^m$, the
 multifunction
 $D^\ast F((\xb,\yb ); (u,v)): \R^m\tto\R^n$, defined by
\begin{equation*}
\gph D^\ast  F((\xb,\yb ); (u,v))=\{(y^*,x^*)\mv (x^*,-y^*)\in N_{\gph F}((\xb,\yb); (u,v))\}
\end{equation*}
is called the {\em directional limiting coderivative} of $F$ at $(\xb,\yb )$ {\em in direction} $(u,v)$.
\end{enumerate}
\end{definition}
 The directional limiting
normal cone and coderivative were introduced by the first author in \cite{Gfr13a} and various properties
of these objects can be found also in \cite{GO3} and in the references therein. Note that $D^\ast  F(\xb,\yb
)=D^\ast  F((\xb,\yb ); (0,0))$ and that $\dom D^\ast  F((\xb,\yb ); (u,v))=\emptyset$ whenever $v\not\in
DF(\xb,\yb)(u)$.

Note that by \cite[Proposition 6.6]{RoWe98} and the definition of the limiting coderivative we have
\begin{equation}\label{EqLimCoderivLimsup}\gph D^*F(\xb,\yb)=\Limsup_{(x,y)\longsetto{{\gph F}}(\xb,\yb)}\gph D^*F(x,y).\end{equation}

If $F$ is single-valued, we can omit the second argument and write $DF(x)$, $\widehat D^*F(x),\ldots$ instead of $DF(x,F(x))$, $\widehat D^*F(x,F(x)),\ldots$. However, be aware that when considering limiting objects at $x$ where $F$ is not continuous, it is not enough to consider only sequences $x_k\to x$ but we must work with sequences $(x_k,F(x_k))\to(x,F(x))$.

\begin{definition}
  Let $U\subset \R^n$ be open and let $F:U\to \R^m$ be a mapping. The {\em B-subdifferential} of $F$ at $x\in U$ is defined as
  \begin{equation}
    \overline{\nabla} F(x):=\{A\mv \exists x_k\to x: \mbox{$F$ is Fr\'echet differentiable at $x_k$ and }A=\lim_{k\to\infty}\nabla F(x_k)\}
  \end{equation}
\end{definition}
Recall that the Clarke Generalized Jacobian is given by $\co\overline{\nabla} F(x)$, i.e., the convex hull of the B-subdifferential.

There exists the following relation between the B-subdifferential and the coderivative of $F$, which states that every element from the B-subdifferential defines a certain subspace contained in the graph of the coderivative.
\begin{proposition}\label{PropBSubdiffCoderiv}
  Let $U\subset \R^n$ be open and let $F:U\to \R^m$ be a mapping.  Let $F$ be continuous at $x\in U$  and let $A\in \overline{\nabla} F(x)$. Then
  \[ (y^*,A^Ty^*)\in\gph D^*F(x)\ \forall y^*\in\R^m.\]
\end{proposition}
\begin{proof}
  Consider $A\in \overline{\nabla} F(x)$ together with some sequence $x_k\to x$ such that $\nabla F(x_k)\to A$ as $k\to\infty$. By \cite[Example 9.25(b)]{RoWe98} we have $T_{\gph F}(x_k,F(x_k))=\{(u,\nabla F(x_k)u)\mv u\in\R^n\}$ and therefore $(y^*,\nabla F(x_k)^Ty^*)\in\gph \widehat D^*F(x_k)$, $\forall y^*\in\R^m$. By passing to the limit, the assertion follows from the definition of the limiting coderivative.
\end{proof}
If the mapping $F:U\to\R^m$ is Lipschitz continuous, by Rademacher's Theorem $F$ is differentiable almost everywhere in $U$ and $\norm{\nabla F(x)}$  is bounded there by the Lipschitz constant of $F$. Thus $\overline{\nabla} F(\xb)\not=\emptyset$ for Lipschitz continuous mappings $F$.

Let $q:\R^n\to\bar\R$ be an extended-real-valued function with the domain and the epigraph
\[\dom q:=\{x\in\R^n\mv q(x)<\infty\},\quad \epi q:=\{(x,\alpha)\in\R^n\times\R\mv \alpha\geq q(x)\}.\]
The {\em(limiting/Mordukhovich)  subdifferential} of  $q$ at $\xb\in \dom q$ is defined geometrically by
\[\partial q(\xb):=\{x^*\in\R^n\mv (x^*,-1)\in N_{\epi q}(\xb,q(\xb))\}.\]
This subdifferential is a general extension of the classical gradient for smooth functions and of the classical
subdifferential of convex ones.

If $q(\xb)$ is finite, define the parametric family of second-order difference quotients for $q$ at $\xb$ for $\xba\in\R^n$ by
\[\Delta_t^2q(\xb,\xba)(w):=\frac{q(\xb+tw)-q(\xb)-t\skalp{\xba,w}}{\frac 12 t^2}\quad\mbox{ with $w\in\R^n$, $t>0$.}\]
The {\em second-order subderivative} of $q$ at $\xb$ for $\xba$ is given by
\[{\rm d^2}q(\xb,\xba)(w)=\liminf_{\AT{t\downarrow 0}{w'\to w}}\Delta_t^2q(\xb,\xba)(w').\]
$q$ is called {\em twice epi-differentiable} at $\xb$ for $\xba$, if the functions $\Delta_t^2q(\xb,\xba)$ epi-converge to ${\rm d^2}q(\xb,\xba)$ as $t\downarrow 0$.

Let us now recall the following regularity notions.
\begin{definition}
  Let $F:\R^n\tto\R^m$ be a mapping and let $(\xb,\yb)\in\gph F$.
  \begin{enumerate}
    \item $F$ is said to be {\em metrically subregular at} $(\xb,\yb)$ if there exists $\kappa\geq 0$  along with some  neighborhood $X$ of $\xb$ such that
    \begin{equation}\label{EqSubreg}
      \dist{x,F^{-1}(\yb)}\leq \kappa\dist{\yb,F(x)}\ \forall x\in X.
    \end{equation}
    The infimum over all $\kappa\geq 0$ such that \eqref{EqSubreg} holds for some neighborhood $X$ is denoted by $\subreg F(\xb,\yb)$.
  \item $F$ is said to be {\em strongly metrically subregular at} $(\xb,\yb)$ if it is metrically subregular at $(\xb,\yb)$ and there exists a neighborhood $X'$ of $\xb$ such that $F^{-1}(\yb)\cap X'=\{\xb\}$.
  \item $F$ is said to be {\em metrically regular around} $(\xb,\yb)$ if there is $\kappa\geq 0$ together with neighborhoods $X$ of $\xb$ and $Y$ of $\yb$ such that
      \begin{equation}\label{EqMetrReg}
      \dist{x,F^{-1}(y)}\leq \kappa\dist{y,F(x)}\ \forall (x,y)\in X\times Y.
    \end{equation}
    The infimum over all $\kappa\geq 0$ such that \eqref{EqMetrReg} holds for some neighborhoods $X,Y$ is denoted by $\reg F(\xb,\yb)$.
  \item $F$ is said to be {\em strongly metrically regular around} $(\xb,\yb)$ if it is metrically regular around $(\xb,\yb)$ and $F^{-1}$ has a single-valued localization around $(\yb,\yb)$, i.e., there are  open neighborhoods $Y'$ of $\yb$, $X'$ of $\xb$ and a mapping $h:Y'\to\R^n$ with $h(\yb)=\xb$ such that $\gph F\cap (X'\times Y')=\{(h(y),y)\mv y\in Y'\}$.
  \end{enumerate}
\end{definition}
It is well-known, see, e.g., \cite{DoRo14} that the property of (strong) metric subregularity for $F$ at $(\xb,\yb)$ is equivalent with the property of {\em (isolated) calmness}  for $F^{-1}$ at $(\yb,\xb)$.
Further, $F$ is metrically regular around $(\xb,\yb)$ if and only if the inverse mapping $F^{-1}$ has the so-called {\em Aubin property} around $(\yb,\xb)$. In this paper we will frequently use the following characterization of strong metric regularity.
\begin{theorem}[{ cf. \cite[Proposition 3G.1]{DoRo14}}]\label{ThStrMetrReg} $F:\R^n\tto\R^m$ is strongly metrically regular around $(\xb,\yb)$ if and only if $F^{-1}$ has a Lipschitz continuous localization $h$  around $(\yb,\xb)$. In this case there holds
\[\reg F(\xb,\yb)=\limsup_{\AT{y,y'\to\yb}{y\not=y'}}\frac {\norm{h(y)-h(y')}}{\norm{y-y'}}.\]
\end{theorem}

In this paper we will also use the  following point-based characterizations of the above regularity properties.
\begin{theorem}\label{ThCharRegByDer}
    Let $F:\R^n\tto\R^m$ be a mapping and let $(\xb,\yb)\in\gph F$.
    \begin{enumerate}
    \item[(i)] (Levy-Rockafellar criterion) $F$ is strongly metrically subregular at $(\xb,\yb)$ if and only if
      \begin{equation}\label{EqLevRoCrit}
        0\in DF(\xb,\yb)(u)\ \Rightarrow u=0.
      \end{equation}
      and in this case one has
      \[\subreg F(\xb,\yb)=\sup\{\norm{u}\mv (u,v)\in\gph DF(\xb,\yb),\ \norm{v}\leq 1\}.\]
    \item[(ii)] (Mordukhovich criterion) $F$ is metrically regular around $(\xb,\yb)$ if and only if
      \begin{equation}
            \label{EqMoCrit} 0\in D^*F(\xb,\yb)(y^*)\ \Rightarrow\ y^*=0.
      \end{equation}
      Further, in this case one has
      \begin{equation}
        \label{EqModMetrReg} {\rm reg\;}F(\xb,\yb)=\sup\{\norm{y^*}\mv (y^*,x^*)\in\gph D^*F(\xb,\yb),\ \norm{x^*}\leq 1\}.
      \end{equation}
    \item[(iii)] $F$ is strongly metrically regular around $(\xb,\yb)$ if and only if
      \begin{equation}
            \label{EqStrictCrit} 0\in D_*F(\xb,\yb)(u)\ \Rightarrow\ u=0
      \end{equation}
      and \eqref{EqMoCrit} holds. In this case one also has
      \begin{equation}\label{EqModStrMetrReg}
        {\rm reg\;}F(\xb,\yb) = \sup\{\norm{u}\mv (u,v)\in\gph D_*F(\xb,\yb),\ \norm{v}\leq 1\}.
      \end{equation}
    \end{enumerate}
\end{theorem}
\begin{proof}
  Statement (i)  follows from \cite[Theorem 4E.1]{DoRo14}. Statement (ii) can be found in \cite[Theorem 3.3]{Mo18}. The criterion for strong metric regularity follows from Dontchev and Frankowska \cite[Theorem 16.2]{DoFra13} by taking into account that the condition $\xb\in \liminf_{y\to\yb}F^{-1}(y)$  appearing in \cite[Theorem 16.2]{DoFra13} can be ensured by the requirement that $F$ is metrically regular which in turn can be characterized by the Mordukhovich criterion.
\end{proof}
For a sufficient condition for metric subregularity based on directional limiting coderivatives we refer to \cite{GO3}.

The properties of (strong) metric regularity and strong metric subregularity are stable under Lipschitzian and calm  perturbations, respectively, cf. \cite{DoRo14}. Further note that the property of (strong) metric regularity holds around all points belonging to the graph of $F$ sufficiently close to the reference point, whereas the property of (strong) metric subregularity
is guaranteed to hold only at the reference point. This leads to the following definition.
\begin{definition}\label{DefLocStrSubReg}
  We say that the mapping $F:\R^n\tto\R^m$ is {\em (strongly) metrically subregular around} $(\xb,\yb)\in\gph F$ if there is a neighborhood $W$ of $(\xb,\yb)$ such that $F$ is (strongly) metrically subregular at every point $(x,y)\in\gph F\cap W$ and we define
  \[\lsubreg F(\xb,\yb):=\limsup_{(x,y)\longsetto{{\gph F}}(\xb,\yb)}\subreg F(x,y)<\infty.\]
  In this case we will also speak about {\em (strong) metric subregularity on a neighborhood}.
\end{definition}
Note that every polyhedral multifunction, i.e., a mapping whose graph is the union of finitely many convex polyhedral sets, is metrically subregular around every point of its graph by Robinson's result \cite{Ro81}. In Section \ref{SecStrMetrSubr}, characterizations of strong metric subregularity on a neighborhood will be investigated.

Next we  introduce the \ssstar sets and mappings.

\begin{definition}\label{DefSemiSmooth}(cf. \cite{GfrOut21}.)
\begin{enumerate}
\item  A set $A\subseteq\R^s$ is called {\em \ssstar} at a point $\xb\in A$ if for all $u\in
    \R^s$ it holds
  \begin{equation}\label{EqSemiSmoothSet}\skalp{x^*,u}=0\quad \forall x^*\in N_A(\xb;u).
  \end{equation}
\item
A set-valued mapping $F:\R^n\tto\R^m$ is called {\em \ssstar} at a point $(\xb,\yb)\in\gph F$, if
$\gph F$ is \ssstar at $(\xb,\yb)$, i.e., for all $(u,v)\in\R^n\times\R^m$ we have
\begin{equation}\label{EqSemiSmooth}
\skalp{u^*,u}=\skalp{v^*,v}\quad \forall (v^*,u^*)\in\gph D^*F((\xb,\yb);(u,v)).
\end{equation}
\end{enumerate}
\end{definition}
The class of semismooth* mappings is rather broad.
We list here two important classes of multifunctions having this property.
\begin{proposition}
  \label{PropSSstar}
  \begin{enumerate}
    \item[(i)]Every mapping whose graph is the union of finitely many closed convex sets is \ssstar at every point of its graph.
    \item[(ii)]Every mapping with closed subanalytic graph is \ssstar at every point of its graph.
  \end{enumerate}
\end{proposition}
\begin{proof}
The first assertion was already shown in \cite[Proposition 3.4, 3.5]{GfrOut21}. As mentioned in \cite[Remark 3.10]{GfrOut21}, the \ssstar property of sets amounts to the notion of semismoothness introduced in \cite{HO01}. It follows thus from \cite[Theorem 2]{Jou07},  that all closed subanalytic sets  are automatically \ssstar and the second statement holds by the definition of \ssstar mappings.
\end{proof}
The statement of Proposition \ref{PropSSstar}(ii) can be considered as the counterpart to \cite{BoDaLe09}, where it is shown that locally Lipschitz tame mappings $F:U\subseteq\R^n\to\R^m$ are semismooth in the sense of Qi and Sun \cite{QiSun93}.
In case of single-valued Lipschitzian mappings the \ssstar property is equivalent with the semismooth property introduced by Gowda \cite{Gow04}, which is weaker than the one in \cite{QiSun93}.

In the above definition the   \ssstar sets and mappings have been defined via directional limiting normal
cones and coderivatives. For our purpose it is convenient to make use of equivalent
characterizations in terms of standard (regular and limiting) normal cones and coderivatives,
respectively.

\begin{proposition}[{cf.\cite[Corollary 3.3]{GfrOut21}}]\label{PropCharSemiSmooth}Let $F:\R^n\tto\R^m$ and $(\xb,\yb)\in \gph F$ be given.
Then the following three statements are equivalent
\begin{enumerate}
\item[(i)] $F$ is \ssstar at $(\xb,\yb)$.
\item[(ii)] For every $\epsilon>0$ there is some $\delta>0$ such that
\begin{equation}\label{EqCharSemiSmoothReg}
\hspace{-1cm}\vert \skalp{x^*,x-\xb}-\skalp{y^*,y-\yb}\vert\leq \epsilon
\norm{(x,y)-(\xb,\yb)}\norm{(x^*,y^*)}\ \forall(x,y)\in \B_\delta(\xb,\yb)\ \forall
(y^*,x^*)\in\gph \widehat D^*F(x,y).\end{equation}
\item[(iii)] For every $\epsilon>0$ there is some $\delta>0$ such that
  \begin{equation}\label{EqCharSemiSmoothLim}
\hspace{-1cm}\vert \skalp{x^*,x-\xb}-\skalp{y^*,y-\yb}\vert\leq \epsilon
\norm{(x,y)-(\xb,\yb)}\norm{(x^*,y^*)}\ \forall(x,y)\in \B_\delta(\xb,\yb)\ \forall
(y^*,x^*)\in\gph D^*F(x,y).
\end{equation}
\end{enumerate}
\end{proposition}

\section{\SCD mappings}

In what follows we denote by $\Z_n$ the metric space of all $n$-dimensional subspaces of $\R^{2n}$ equipped with the metric
\[d_\Z(L_1,L_2):=\norm{P_1-P_2}\]
where $P_i$ is the symmetric $2n\times 2n$ matrix representing the orthogonal projection on $L_i$, $i=1,2$.

Sometimes we will also work with bases for the subspaces $L\in\Z_n$. Let $\M_n$ denote the collection of all $2n\times n$ matrices with full  rank $n$ and for $L\in \Z_n$ we define
\[\M(L):=\{Z\in \M_n\mv \rge Z =L\},\]
i.e., the columns of $Z\in\M(L)$ are a basis for $L$. Further we denote by $\M^{\rm orth}(L)$ the set of all matrices $Z\in \M(L)$ with $Z^TZ=I$, i.e., the columns of $Z$ are an orthogonal basis for $L$. Recall that, given any matrix $\bar Z\in \M(L)$ ( or $\bar Z\in \M^{\rm orth}(L)$), there holds
\[\M(L)=\{\bar ZB\mv B\mbox{ nonsingular $n\times n$ matrix}\}\quad (\M^{\rm orth}(L)=\{\bar ZB\mv B\mbox{ orthogonal $n\times n$ matrix}\}).\]
Further recall that the  $2n\times 2n$ matrix $P$, representing the orthogonal projection on some $L\in\Z_n$, admits the representations
\begin{equation}\label{EqProjRepr}P=Z(Z^TZ)^{-1}Z^T,\ Z\in \M(L)\quad\mbox{and}\quad P=ZZ^T,\ Z\in \M^{\rm orth}(L).\end{equation}
\begin{lemma}\label{LemCompact}
  \begin{enumerate}
    \item[(i)]Let $Z_k\in\M_n$ be a sequence converging to some $Z\in \M_n$. Then $\rge Z_k$ converges in $\Z_n$ to $\rge Z\in\Z_n$.
    \item[(ii)]Let $L_k\in\Z_n$ be a sequence converging to $L\in\Z_n$. Then there is a sequence $Z_k \in\M(L_k)$ converging to some $Z\in \M(L)$.
    \item[(iii)]Let $A_k$ be a sequence of nonsingular $2n\times 2n$ matrices converging to a nonsingular matrix $A$ and let $L_k\in\Z_n$ be a sequence converging to $L\in\Z_n$. Then
    $\lim_{k\to\infty}d_\Z(A_kL_k,AL)=0$.
    \item[(iv)]The metric space $\Z_n$ is compact.
    \item[(v)]Let $L_k\in\Z_n$ be a sequence and let $L\in\Z_n$. Then $L_k$ converges to $L$ in $\Z_n$, i.e., $\lim_{k\to\infty}d_\Z(L_k,L)=0$, if and only if $\lim_{k\to\infty}L_k=L$ in the sense of Painlev\'e-Kuratowski convergence.
  \end{enumerate}
\end{lemma}
\begin{proof}
  The first statement follows immediately from $Z(Z^TZ)^{-1}Z^T=\lim_{k\to\infty}Z_k(Z_k^TZ_k)^{-1}Z_k^T$ together with \eqref{EqProjRepr}. In order to prove (ii), choose $Z_k\in \M^{\rm orth}(L_k)$  and $Z\in \M^{\rm orth}(L)$. Then $ZZ^T=\lim_{k\to\infty}Z_kZ_k^T$ due to $L_k\longsetto{\Z_n}L$ and consequently $Z=ZZ^TZ=\lim_{k\to\infty}\tilde Z_k$ with $\tilde Z_k:=Z_k(Z_k^TZ)$. Hence, for sufficiently large $k$ we have $\tilde Z_k\in \M_n$ and $\tilde Z_k\in \M(L_k)$ follows. This proves (ii) and (iii) follows from (ii) and (i).
  In order to prove the compactness of $\Z_n$, consider a sequence $L_k\in\Z_n$ together with basis matrices $Z_k\in \M^{\rm orth}(L_k)$. By possibly passing to a subsequence we may assume that $Z_k$ converges to  some $Z$. Since $Z^TZ=\lim_{k\to\infty}Z_k^TZ_k=I$, we conclude $Z\in\M_n$ and $\rge Z=\lim_{k\to\infty}L_k\in\Z_n$ by (i). Hence the metric space $\Z_n$ is (sequentially) compact. Finally, by \cite[Example 5.35]{RoWe98} there holds $L_k\to L$ in the sense of Painlev\'e-Kuratowski convergence if and only if the projections $P_{L_k}$ on $L_k$ converge graphically to the projection $P_L$ on $L$. Since the projections on subspaces in $\Z_n$ are linear mappings with norm equal to $1$, graphical convergence of $P_{L_k}$ to $P_L$ is equivalent to uniform convergence $\lim_{k\to\infty}\norm{P_{L_k}-P_L}=0$ by \cite[Theorems 5.43, 5.44]{RoWe98}.
\end{proof}
We treat every element of $\R^{2n}$ as a column vector. In order to keep our notation simple we write $(u,v)$ instead of $\myvec{u\\ v}\in\R^{2n}$ when this does not lead to confusion. In order to refer to the components of the vector $z=\myvec{u\\ v}$ we set $\pi_1(z):=u,\ \pi_2(z):=v$.

Let $L\in\Z_n$ and consider $Z\in \M(L)$, which can be written in the form $Z=\myvec{A\\B}$. But we will rather write it as $Z=(A,B)$; thus  $\rge(A,B):=\{(Au,Bu)\mv u\in\R^n\}\doteq \Big\{\myvec{Au\\Bu}\mv u\in\R^n\Big\}=L$. Similarly as before, we will also use $\pi_1(Z):=A$, $\pi_2(Z):=B$ for referring to the two $n\times n$ parts of $Z$.

Further, for every $L\in \Z_n$ we define
\begin{align}\label{EqDualSubspace}
  L^*&:=\{(-v^*,u^*)\mv (u^*,v^*)\in L^\perp\},
\end{align}
where $L^\perp$ denotes as usual the orthogonal complement of $L$. Note that
\[(L^*)^\perp=\{(v,u)\mv \skalp{v,-v^*}+\skalp{u,u^*}=0\ \forall (u^*,v^*)\in L^\perp\}=\{(v,u)\mv (u,-v)\in (L^\perp)^\perp\}\]
and therefore
\begin{align*}
  (L^*)^*=\{(u,v)\mv (v,-u)\in (L^*)^\perp\}=\{(u,v)\mv (-u,-v)\in L\}=L.
\end{align*}
We denote by $S_n$ the $2n\times 2n$ orthogonal matrix
\[S_n:=\left(\begin{matrix}0&-I\\I&0\end{matrix}\right),\]
so that $L^*=S_nL^\perp$. If $P$ represents the orthogonal projection on $L$ then $I-P$ is the orthogonal projection on $L^\perp$ and $S_n(I-P)S_n^T$ is the orthogonal projection on $L^*$. Given two subspaces $L_1,L_2\in \Z_n$ with orthogonal projections $P_1,P_2$, we obtain
\[d_\Z(L_1^*,L_2^*)=\norm{S_n(I-P_1)S_n^T-S_n(I-P_2)S_n^T}=\norm{S_n(I-P_1-(I-P_2))S_n^T}=\norm{P_1-P_2}=d_\Z(L_1,L_2).\]
Thus the mapping $L\mapsto L^*$ defines an isometry on $\Z_n$ and a sequence $(L_k)$ converges in $\Z_n$ to some $L$ if and only if the sequence $(L_k^*)$ converges to $L^*$.

Consider the following relation between the graphical derivative and differentiability in case of single-valued mappings.
\begin{lemma}\label{LemDiffSingleValued}
  Consider $f:U\to\R^n$ with $U\subseteq \R^n$ open and a point $x\in U$. Then one has:
  \begin{enumerate}
    \item[(i)]If $f$ is Fr\'echet differentiable at $x$, then $DF(x)$ is a single-valued linear mapping, $DF(x)(u)=\nabla f(x)u$, $u\in\R^n$, and consequently $T_{\gph f}(x,f(x))=\rge(I,\nabla f(x))\in\Z_n$.
    \item[(ii)]Conversely, if $T_{\gph f}(x,f(x))\in\Z_n$ and $f$ is calm at $x$, i.e., there is some $\kappa\geq 0$ such that the estimate $\norm{f(x')-f(x)}\leq \kappa\norm{x'-x}$ holds for all $x'$ sufficiently close to $x$, then $f$ is Fr\'echet differentiable at $x$.
  \end{enumerate}
\end{lemma}
\begin{proof}
  The statement (i) follows immediately from \cite[Exercise 9.25]{RoWe98}. In order to show (ii), we first prove that there is an  $n\times n$ matrix $A$ such that
  $T_{\gph f}(x,f(x))=\rge(I,A)$. Considering any $Z\in \M(T_{\gph f}(x,f(x)))$, we will show that $B:=\pi_1(Z)$ is nonsingular. Assuming on the contrary that $B$ is singular, there is some $p\not=0$ with $Bp=0$. Then $v:=\pi_2(Z)p\not=0$ because otherwise $Zp=0$ which is not possible. Hence $(0,v)=Zp\in T_{\gph f}$ and there exists sequences $t_k\downarrow 0$ and $(u_k,v_k)\to (u,v)$ such that $f(x)+t_kv_k=f(x+t_ku_k)$ $\forall k$ implying $t_k\norm{v_k}=\norm{f(x+t_ku_k)-f(x)}\leq \kappa t_k\norm{u_k}$ and $\norm{v}=\lim_{k\to\infty}\norm{u_k}=0$, a contradiction. Hence $B$ is nonsingular and we obtain $T_{\gph f}(x,f(x))=\rge (B,\pi_2(Z))=\rge (I,\pi_2(Z)B^{-1})$ proving our claim with $A=\pi_2(Z)B^{-1}$. Hence $Df(x)u=Au$, $u\in\R^n$ and the assertion follows once more from \cite[Exercise 9.25]{RoWe98}.
\end{proof}
Note that, when $f:U\to \R^n$, $U\subseteq\R^n$ open, is Fr\'echet differentiable at $u\in U$, then we even have $T_{\gph f}(u,f(u))={\rm Lim}_{t\downarrow 0}t^{-1}(\gph f-(u,f(u)))$.

We now introduce  new generalized derivatives for set-valued mappings. We confine ourselves to the particular case $F:\R^n\tto\R^n$ and, as we will see in the sequel, this restriction still permits a considerable number of applications.

\begin{definition}\label{DefSCDProperty}
 Consider a mapping $F:R^n\tto\R^n$.
  \begin{enumerate}
    \item We say that $F$ is {\em graphically  smooth of dimension $n$} at $(x,y)\in \gph F$, if $T_{\gph F}(x,y)=\gph DF(x,y)\in \Z_n$. Further we denote by $\OO_F$ the set of all points where $F$ is graphically smooth of dimension $n$.
    \item Associate with $F$ the four mappings $\widehat\Sp F$, $\widehat\Sp^* F$, $\Sp F$, $\Sp^* F$, all of which map $\gph F\tto \Z_n$ and are given by
    \begin{align*}\widehat\Sp F(x,y)&:=\begin{cases}\{\gph DF(x,y)\}& \mbox{if $(x,y)\in\OO_F$,}\\
    \emptyset&\mbox{else,}\end{cases}\\
    \widehat\Sp^* F(x,y)&:=\begin{cases}\{\gph DF(x,y)^*\}& \mbox{if $(x,y)\in\OO_F$,}\\
    \emptyset&\mbox{else,}\end{cases}\\
    \Sp F(x,y)&:=\Limsup_{(u,v)\longsetto{{\gph F}}(x,y)} \widehat\Sp F(u,v) \\
    &=\{L\in \Z_n\mv \exists (x_k,y_k)\longsetto{{\OO_F}}(x,y):\ \lim_{k\to\infty} d_\Z(L,\gph DF(x_k,y_k))=0\},\\
    \Sp^* F(x,y)&=\Limsup_{(u,v)\longsetto{{\gph F}}(x,y)} \widehat\Sp^* F(u,v)\\
    &=\{L\in \Z_n\mv \exists (x_k,y_k)\longsetto{{\OO_F}}(x,y):\ \lim_{k\to\infty} d_\Z(L,\gph DF(x_k,y_k)^*)=0\}.
    \end{align*}
    \item \begin{enumerate}
    \item We say that $F$ has the {\em\SCD} (subspace containing derivative) {\em property at} $(x,y)\in\gph F$, if $\Sp^*F(x,y)\not=\emptyset$.
     \item We say that $F$ has the \SCD property {\em around} $(x,y)\in\gph F$, if there is a neighborhood $W$ of $(x,y)$ such that $F$ has the \SCD property at every $(x',y')\in\gph F\cap W$.
     \item Finally, we call $F$ an {\em \SCD mapping} if
    $F$ has the \SCD property at every point of its graph.
    \end{enumerate}
  \end{enumerate}
\end{definition}
Apart from the collections $\Sp F(x,y)$ and $\Sp^*F(x,y)$ of subspaces we will sometimes use the unions
\begin{equation}
  \bigcup\Sp F(x,y):=\bigcup_{L\in \Sp F(x,y)}L, \quad \bigcup\Sp^* F(x,y):=\bigcup_{L\in \Sp^* F(x,y)}L.
\end{equation}
\begin{remark}\label{RemRegCoder}
  By definition of the regular coderivative there holds
  \[\gph \widehat D^*F(x,y) =\gph DF(x,y)^*,\ (x,y)\in\OO_F.\]
\end{remark}
\begin{remark}\label{RemSymmetry}
  Since $L\mapsto L^*$ is an isometry on $\Z_n$ and $(L^*)^*=L$, we have
  \[\Sp^* F(x,y)=\{L^*\mv L\in \Sp F(x,y)\},\quad \Sp F(x,y)=\{L^*\mv L\in \Sp^* F(x,y)\}.\]
  Hence, $F$ has the \SCD property at $(x,y)\in\gph F$ if and only if $\Sp F(x,y)\not=\emptyset$.
\end{remark}

Since we consider convergence in the compact metric space $\Z_n$, we obtain readily the following result.

\begin{lemma}\label{LemSCDproperty}
 A mapping $F:\R^n\tto\R^n$ has the \SCD property at $(x,y)\in\gph F$ if and only if $(x,y)\in\cl \OO_F$. Further, $F$ is an \SCD mapping if and only if $\cl \OO_F=\cl \gph F$, i.e., $F$ is graphically smooth of dimension $n$ at the points of a dense subset of its graph.
\end{lemma}
The name ''\SCD property'' is motivated by the following statement.
\begin{lemma}\label{LemSCDname}
  Let $F:\R^n\tto\R^n$ and let $(x,y)\in\gph F$. Then $\bigcup\Sp^*F(x,y)\subseteq \gph D^*F(x,y)$.
\end{lemma}
\begin{proof}
  Let $L\in \Sp^*F(x,y)$ and consider a sequence $(x_k,y_k,L_k)\to(x,y,L)$ with $(x_k,y_k)\in\OO_F$ and $L_k:=(\gph DF(x_k,y_k))^*\in \widehat \Sp^*F(x_k,y_k)$. By Remark \ref{RemRegCoder} we have $L_k = \gph \widehat D^*F(x_k,y_k)$. Consider  $Z_k\in \M^{\rm orth}(L_k)$. By possibly passing to a subsequence the matrices $Z_k$ converge to some $Z$ and $L=\rge Z$ by Lemma \ref{LemDiffSingleValued}. Taking into account $Z_kp\in L_k= \gph \widehat D^*F(x_k,y_k)$, we obtain $Zp=\lim_{k\to\infty}Z_kp\in \gph D^*F(x,y)$ $\forall p\in\R^n$ by the Definitions \ref{DefVarGeom}, \ref{DefGenDeriv} showing that $L\subseteq \gph D^*F(x,y)$. Since this holds for every $L\in\Sp^* F(x,y)$, the assertion follows.
\end{proof}

We will now show that the primal subspaces $L\in\Sp F(x,y)$ also belong to the graph of some suitable generalized derivative mapping. Consider the following definition.

\begin{definition}\label{DefLimGrDer}
  \begin{enumerate}
    \item Let $A\subset \R^n$ and let $\xb\in A$. The {\em outer limiting tangent cone} to $A$ at $\xb$ is defined as
    \begin{equation}\label{EqLimTanCone}
      T^\sharp_A(\xb):=\Limsup_{x\setto{{A}}\xb} T_A(x)= \Limsup_{x\setto{{A}}\xb}\Big(\Limsup_{t\downarrow 0} \frac {A-x}t\Big)
    \end{equation}
    \item Consider a  multifunction $F:\R^n\tto\R^m$ and let $(\xb,\yb)\in \gph F$. The {\em outer limiting graphical derivative} of $F$ at $(\xb,\yb)$ is the multifunction
    $D^\sharp F(\xb,\yb):\R^n\tto\R^m$ given by
    \[\gph D^\sharp F(\xb,\yb)= T^\sharp_{\gph F}(\xb,\yb).\]
  \end{enumerate}
\end{definition}
\begin{remark}
  Comparing \eqref{EqLimTanCone} with the definition of the paratingent cone $T^P_A(\xb)$ it follows that $T^\sharp_A(\xb)\subseteq T^P_A(\xb)$ and therefore $D^\sharp F(\xb,\yb)(u)\subseteq D_* F(\xb,\yb)(u)$, $u\in\R^n$.
\end{remark}

Using similar arguments as in the proof of Lemma \ref{LemSCDname} one obtains the following result.
\begin{lemma}\label{LemSCDname1}
  Let $F:\R^n\tto\R^n$ and let $(x,y)\in\gph F$. Then $\bigcup\Sp F(x,y)\subseteq \gph D^\sharp F(x,y)$.
\end{lemma}

For single-valued mappings the constructions of Definition \ref{DefSCDProperty} are related to  the B-subdifferential.
\begin{lemma}\label{LemSCDSingleValued}
  Let $U\subset \R^n$ be open and let $f:U\to\R^n$ be continuous. Then for every $x\in U$ there holds
  \begin{align}
    \label{EqGenBSubdiff1}&\Sp f(x):=\Sp(x,f(x))\supseteq \{\rge(I,A)\mv A\in\overline{\nabla} f(x)\},\\
    \label{EqGenBSubdiff2}&\Sp^* f(x):=\Sp^*(x,f(x))\supseteq \{\rge(I,A^T)\mv A\in\overline{\nabla} f(x)\}.
  \end{align}
  If $f$ is Lipschitz continuous near $x$, these inclusions hold with equality and $f$ has the \SCD property around $x$.
\end{lemma}
\begin{proof}
   Consider $x\in U$ and $A\in \overline{\nabla} f(x)$ together with sequences $x_k\to x$ and $\nabla f(x_k)\to A$. Then for each $k$ we have $T_{\gph f}(x_k,f(x_k))= \rge(I,\nabla f(x_k))=:L_k\in\Z_n$ by Lemma \ref{LemDiffSingleValued} implying that $(x_k,f(x_k))\in\OO_f$ and $\widehat\Sp f(x_k)=\{L_k\}$. Thus the subspaces $L_k$ converge in $\Z_n$ to $\rge(I,A)\in \Sp f(x)$ by Lemma \ref{LemCompact}(i). This proves \eqref{EqGenBSubdiff1}. By taking into account the identity $\rge(I,A)^\perp=\rge(-A^T,I)$, it follows that $\rge(I,A)^*=\rge (I,A^T)$ verifying \eqref{EqGenBSubdiff2}. Now assume that $f$ is Lipschitzian near $x$ and consider $L\in\Sp f(x)$ together with a sequence $(x_k,f(x_k))\longsetto{{\OO_f}}(x,f(x))$ such that $L_k:=T_{\gph f}(x_k,f(x_k))\longsetto{{\Z_n}} L$. By Lemma \ref{LemDiffSingleValued} we conclude that $f$ is differentiable at $x_k$ and $L_k=\rge(I,\nabla f(x_k))$. By Lipschitz continuity of $f$ the derivatives $\nabla f(x_k)$ are bounded. Hence, by possibly passing to a subsequence, we can assume that $\nabla f(x_k)$ converges to some $A\in\overline{\nabla} f(x)$ and $L_k\longsetto{{\Z_n}} \rge(I,A)$ follows. This proves equality in \eqref{EqGenBSubdiff1} and equality in \eqref{EqGenBSubdiff2} easily follows from the identity $\rge(I,A)^*=\rge(I,A^T)$. Since $\overline{\nabla} f(x)\not=\emptyset$ for Lipschitz continuous mappings, the \SCD property at $x$ is established. This also holds for every point sufficiently  close to $x$ and thus $f$ has the \SCD property even around $x$.
\end{proof}

\begin{remark}
 In particular, every  Lipschitz continuous mapping $f:U\to\R^n$ with $U\subseteq \R^n$ open is an \SCD mapping. However, the converse is not true: Consider the function $f(x)=\sqrt{\vert x\vert}$ which is an \SCD mapping but not Lipschitz continuous.
\end{remark}

\begin{lemma}\label{LemLimSupSp}Consider a mapping  $F:\R^n\tto\R^n$ and let $(x,y)\in \gph F$. Then
\[\Sp F(x,y)=\Limsup_{(u,v)\longsetto{{\gph F}}(x,y)}\Sp F(u,v),\quad \Sp^* F(x,y)=\Limsup_{(u,v)\longsetto{{\gph F}}(x,y)}\Sp^* F(u,v).\]
\end{lemma}
\begin{proof}
We prove only the first equation. The inclusion $\Sp F(x,y)\subseteq\Limsup_{(u,v)\longsetto{{\gph F}}(x,y)}\Sp F(u,v)= :S$ follows easily from  the definition of $\Sp F(x,y)$ together with $\widehat \Sp F(u,v)\subseteq \Sp F(u,v)$, $(u,v)\in\gph F$. In order to show the reverse inclusion, consider $L\in S$ together with sequences $(u_k,v_k)\longsetto{{\gph F}}(x,y)$ and $L_k\in \Sp F(u_k,v_k)$ with $L_k\longsetto{{\Z_n}}L$. By definition, for every $k$ we can find $(u_k',y_k')\in\OO_F$ and $L_k'\in \widehat\Sp F(u_k,v_k))$ such that $\norm{(u_k,v_k)-(u_k',v_k')}\leq\frac 1k$ and $d_\Z(L_k,L_k')\leq\frac 1k$. Thus $(u_k',v_k')\longsetto{{\OO_F}}(x,y)$ and $L_k'\longsetto{{\Z_n}}L$ verifying $L\in \Sp F(x,y)$.
\end{proof}

We now provide some calculus rules.
\begin{proposition}\label{PropComposit}
  Given a mapping $G:\R^n\tto\R^n$ and a mapping $\Phi:\R^{2n}\to\R^{2n}$, consider the mapping $F:\R^n\tto\R^n$ given by
  \[\gph F=\{(x,y)\mv \Phi(x,y)\in\gph G\}.\]
  Then for every $(x,y)\in\gph F$ such that $\Phi$ is continuously differentiable in some neighborhood of $(x,y)$ and $\nabla \Phi(x,y)$ is nonsingular, there holds
  \begin{align}
   \label{EqCompositPr} & \Sp F(x,y)=\nabla \Phi(x,y)^{-1} \Sp G(\Phi(x,y))(:= \{\nabla \Phi(x,y)^{-1}L\mv L\in \Sp G(\Phi(x,y))\}),\\
   \label{EqCompositDu} & \Sp^* F(x,y)=S_n\nabla \Phi(x,y)^TS_n^T\Sp^* G(\Phi(x,y)).
  \end{align}
\end{proposition}
\begin{proof}By the classical Inverse Function Theorem, there is some open neighborhood $W$ of $(x,y)$ such that $\Phi$ is a one-to-one mapping from $W$ to the open neighborhood $\tilde W:=\Phi(W)$ of $\Phi(x,y)$ and $\nabla\Phi(x',y')$ is nonsingular for every $(x',y')\in W$.
  By \cite[Exercise 6.7]{RoWe98} we have $T_{\gph F}(x',y')=\nabla \Phi(x',y')^{-1}T_{\gph G}(\Phi(x',y'))$ for all $(x',y')\in \gph F\cap W$ and it follows that $\OO_G\cap\tilde W=\Phi(\OO_F\cap W)$. Consider $L\in\Sp F(x,y)$ together with sequences $(x_k,y_k)\longsetto{\OO_F}(x,y)$ and $L_k\in\widehat\Sp F(x_k,y_k)$ converging to $L$. Then for all $k$ sufficiently large we have $T_{\gph G}(\Phi(x_k,y_k))=\nabla \Phi(x_k,y_k)T_{\gph F}(x_k,y_k)=\nabla \Phi(x_k,y_k)L_k\in \widehat \Sp G(\Phi(x_k,y_k)$ showing $\nabla \Phi(x,y)L=\lim_{k\to\infty}\nabla \Phi(x_k,y_k)L_k\in \Sp G(\Phi(x,y))$ by Lemma \ref{LemCompact}(iii). This proves that $\Sp F(x,y)\subseteq \nabla \Phi(x,y)^{-1}\Sp G(\Phi(x,y))$.

  To show the reverse inclusion, consider $L\in \Sp G(\Phi(x,y))$ together with sequences $z_k\longsetto{\OO_G\cap\tilde W}\Phi(x,y)$ and $L_k\in\widehat \Sp G(z_k)$ with $L_k\to L$.
  It follows that the sequence $(x_k,y_k):=\Phi^{-1}(z_k)\cap W$ converges to $(x,y)$ and $T_{\gph F}(x_k,y_k)=\nabla \Phi(x_k,y_k)^{-1}L_k\in \widehat \Sp F(x_k,y_k)$ implying $\nabla\Phi(x,y)^{-1}L\in \Sp F(x,y)$ by Lemma \ref{LemDiffSingleValued}(iii). Hence $\nabla \Phi(x,y)^{-1}\Sp G(\Phi(x,y))\subseteq\Sp F(x,y)$ and equation \eqref{EqCompositPr} follows. To prove the equation \eqref{EqCompositDu}, just use Remark \ref{RemSymmetry} together with the fact that for any $L\in\Z_n$ we have $(\nabla\Phi(x,y)^{-1}L)^\perp=\nabla\Phi(x,y)^TL^\perp$ implying \[(\nabla\Phi(x,y)^{-1}L)^*=S_n\nabla\Phi(x,y)^TL^\perp=S_n\nabla\Phi(x,y)^TS_n^TL^*.\]

\end{proof}

\begin{proposition}\label{PropSum}
  Let $F:\R^n\tto\R^n$ have the \SCD property at $(x,y)\in\gph F$ and let $h:U\to\R^n$ be continuously differentiable at $x\in U$ where $U\subseteq\R^n$ is open. Then $F+h$ has the \SCD property at $(x,y+h(x))$ and
  \begin{align}\label{EqSCDSum1}\Sp (F+h)(x,y+h(x))=\Big\{\left(\begin{matrix}I&0\\\nabla h(x)&I\end{matrix}\right)L\mv L\in \Sp F(x,y)\Big\}\\
  \label{EqSCDSum2}\Sp^* (F+h)(x,y+h(x))=\Big\{\left(\begin{matrix}I&0\\\nabla h(x)^T&I\end{matrix}\right)L\mv L\in \Sp^* F(x,y)\Big\}
  \end{align}
\end{proposition}
\begin{proof}
    We have $\gph (F+h)=\{(u,v+h(u))\mv (u,v)\in \gph F\}=\{(x,y)\mv (x,y-h(x))\in\gph F\}$ and the assertion follows from Proposition \ref{PropComposit} with $\Phi(x,y)=(x,y-h(x))$.
\end{proof}
Next, let us proceed to the large class of graphically Lipschitzian mappings.
\begin{definition}[cf.{\cite[Definition 9.66]{RoWe98}}]\label{DefGraphLip}A mapping $F:\R^n\tto\R^m$ is {\em graphically Lipschitzian of dimension $d$} at $(\xb,\yb)\in\gph F$ if there is an open neighborhood $W$ of $(\xb,\yb)$ and a one-to-one mapping $\Phi$ from $W$ onto an open subset of $\R^{n+m}$ with $\Phi$ and $\Phi^{-1}$ continuously differentiable, such that $\Phi(\gph F\cap W)$  is the graph of a Lipschitz continuous mapping $f:U\to\R^{n+m-d}$, where $U$ is an open set in $\R^d$.
\end{definition}
In what follows we will refer to the mapping $\Phi$ as {\em transformation mapping}.
\begin{proposition}\label{PropGraphLipSCD}
  Assume that $F:\R^n\tto\R^n$ is graphically Lipschitzian of dimension $n$ at $(\xb,\yb)\in\gph F$ with transformation mapping $\Phi$. Then $F$ has the \SCD property around $(\xb,\yb)$ and for every $(x,y)\in\gph F$, sufficiently close to $(\xb,\yb)$, one has
  \begin{subequations}\label{EqGraphLipSCD}
  \begin{align}
  \label{EqGraphLipSCD1}&\Sp F(x,y)=\nabla \Phi(x,y)^{-1} \Sp f(u)=\Big\{\rge \left[\nabla \Phi(x,y)^{-1}\myvec{I\\ B}\right]\mv B\in\overline{\nabla} f( u)\Big\},\\
  \label{EqGraphLipSCD2}&\Sp^*F(x,y)=S_n\nabla \Phi(x,y)^TS_n^T \Sp^*f(u)=\Big\{\rge \left[S_n\nabla \Phi(x,y)^TS_n^T\myvec{I\\ B^T}\right]\mv B\in\overline{\nabla} f(u)\Big\},
  \end{align}
  \end{subequations}
  where $f$ is as in Definition \ref{DefGraphLip} and $u:=\pi_1(\Phi(x,y))$.
\end{proposition}
\begin{proof} Follows from Proposition  \ref{PropComposit} together with Lemma \ref{LemSCDSingleValued}.
\end{proof}
\begin{remark}\label{RemProtoDiff}Note that for a graphically Lipschitizian mapping $F$ with transformation mapping $\Phi$ we have $\Phi(\OO_F\cap W)=\OO_f$  by the proof of Proposition \ref{PropComposit}, where $W$ and $f$ are as in Definition \ref{DefGraphLip}. At points $(u,f(u))\in \OO_f$ the mapping $f$ is Fr\'echet differentiable at $u$ by Lemma \ref{LemDiffSingleValued}(ii) and therefore $T_{\gph f}(u,f(u))={\rm Lim}_{t\downarrow 0}t^{-1}(\gph f- (u,f(u)))$. Since the graphs of $F$ and $f$ coincide locally up to a change of coordinates, we may conclude that
$T_{\gph F}(x,y)={\rm Lim}_{t\downarrow 0}t^{-1}(\gph F-(x,y))$, $(x,y)\in \OO_F\cap W$, i.e., $F$ is proto-differentiable at these points, cf. \cite[Section 8.H]{RoWe98}.\end{remark}
\begin{corollary}\label{CorStrongMetrReg}Let $F:\R^n\tto\R^n$ and let $(\xb,\yb)\in\gph F$ be given. Suppose that there is an open neighborhood $V$ of $\xb$ and a continuously differentiable mapping $h:V\to\R^n$ such that $F+h$ is strongly metrically regular around $(\xb,\yb+h(\xb))$. Then $F$ is graphically Lipschitzian of dimension $n$ with transformation mapping $\Phi(x,y)=(y+h(x),x)$.
Therefore $F$ has the \SCD property around $(\xb,\yb)$ and for every $(x,y)\in\gph F$ sufficiently close to $(\xb,\yb)$ one has
\begin{subequations}\label{EqSpStrMetrReg}
  \begin{align}\label{EqSpStrMetrReg1}\Sp F(x,y)&=\{\rge(B, I-\nabla h(x)B)\mv B\in\overline{\nabla}(F+h)^{-1}(y+h(x))\}\\
  \nonumber&=\left(\begin{matrix}0&I\\I&-\nabla h(\xb)\end{matrix}\right)\Sp(F+h)^{-1}(\yb+h(\xb),\xb),\\
  \label{EqSpStrMetrReg2}\Sp^*F(x,y)&=\{\rge(B^T, I-\nabla h(x)^TB^T)\mv B\in\overline{\nabla}(F+h)^{-1}(y+h(x))\}\\
  \nonumber&=\left(\begin{matrix}0&I\\I&-\nabla h(x)^T\end{matrix}\right)\Sp^*(F+h)^{-1}(y+h(x),x).
  \end{align}
\end{subequations}
\end{corollary}
\begin{proof}
  By Theorem \ref{ThStrMetrReg} there are  open neighborhoods $U$ of $\yb+h(\xb)$, $W'$ of $(\yb+h(\xb),\xb)$ and a Lipschitz continuous mapping $f:U\to \R^n$ such that $\gph f=\gph (F+h)^{-1}\cap W'$. Since $\gph (F+h)^{-1}=\{(y+h(x),x)\mv  (x,y)\in\gph F,\ x\in V\}$,  $F$ is graphically Lipschitzian of dimension $n$ and \eqref{EqSpStrMetrReg}  follows from Proposition \ref{PropGraphLipSCD}  by taking into account that
  \[\nabla \Phi(x,y)=\left(\begin{matrix}\nabla h(x)&I\\I&0  \end{matrix} \right),\ \nabla \Phi(x,y)^{-1}=\left(\begin{matrix}0&I\\I&-\nabla h(x)  \end{matrix} \right),\
  \ S_n \nabla \Phi(x,y)^T S_n^T=-\left(\begin{matrix}0&I\\I&-\nabla h(x)^T\end{matrix}\right).\]
\end{proof}
Some examples of graphically Lipschitzian mappings $F:\R^n\tto\R^n$ of dimension $n$ were already given in \cite{Ro85,PolRo96}. Next, we extend these examples and give an explicit description of the subspaces contained in $\Sp F(x,y)$ and $\Sp^*F(x,y)$, respectively.
Recall that a mapping $F:\R^n\tto\R^n$ is said to be monotone if
\[\skalp{y_1-y_2,x_1-x_2}\geq 0\quad \mbox{ for all $(x_i,y_i)\in \gph F,\ i=1,2$.}\]
It is maximally monotone if, in addition,  there holds $\gph F=\gph T$ for every monotone mapping $T:\R^n\tto\R^n$ with $\gph F\subset\gph T$.
Next we define several types of local monotonicity.
\begin{definition}\label{DefLocMon}
Let $F:\R^n\tto\R^n$ and let $(x,y)\in\gph F$. We say the following:
\begin{enumerate}
  \item $F$ is {\em locally monotone} at $(x,y)$ if there is an open  neighborhood $X\times Y$ of $(x,y)$ such that
        \begin{equation}\label{EqLocMon}\skalp{y_1-y_2,x_1-x_2}\geq 0\quad \mbox{ for all $(x_i,y_i)\in \gph F\cap (X\times Y),\ i=1,2$}.\end{equation}
        It is {\em locally maximally monotone} if, in addition, there holds $\gph F\cap (X\times Y)=\gph T\cap (X\times Y)$ for every monotone mapping $T:\R^n\tto\R^n$ with $\gph F\cap (X\times Y)\subseteq \gph T$.
  \item $F$ is {\em locally (maximally) hypomonotone} at $(x,y)$ if $\gamma I+F$ is locally (maximally) monotone at $(x,\gamma x+y)$ for some $\gamma\geq 0$.
\end{enumerate}
\end{definition}
Related with maximally monotone operators are the so-called firmly nonexpansive mappings.
\begin{definition}
\begin{enumerate}
  \item A mapping $f:\R^n\to\R^n$ is called {\em firmly nonexpansive} if  $\skalp{f(u_1)-f(u_2),u_1-u_2}\geq \norm{f(u_1)-f(u_2)}^2$, $u_1,u_2\in\R^n$.
  \item An $n\times n$ matrix $B$ is called firmly nonexpansive, if the linear mapping $u\to Bu$ is firmly nonexpansive, i.e., $\skalp{Bv,v}\geq \norm{Bv}^2$, $v\in\R^n$.
\end{enumerate}
\end{definition}
Note that  an $n\times n$ matrix $B$ is firmly nonexpansive if and only if $\norm{2B-I}\leq 1$, see, e.g., \cite[Fact 1.1]{BauMofWa12}. Further, a firmly nonexpansive matrix $B$ is positive semidefinite and satisfies $\norm{B}\leq 1$ and, when $B$ is symmetric, these conditions are also sufficient for $B$ being firmly nonexpansive.
\begin{proposition}
Let $F:\R^n\tto\R^n$ be locally maximally monotone at $(\xb,\yb)$. Then $F$ is graphically Lipschtzian of dimension $n$ at $(\xb,\yb)$ with transformation mapping $\Phi(x,y)=(x+y,x)$ and consequently $F$ has the \SCD property around $(\xb,\yb)$. Further, for every $(x,y)\in\gph F$ sufficiently close to $(\xb,\yb)$  and for every subspace $L\in \Sp F(x,y)$ there is a firmly nonexpansive $n\times n$ matrix $B$ such that $L=\rge(B,I-B)$ and $L^*=\rge(B^T, I-B^T)$.
 \end{proposition}
 \begin{proof}
Let $F_{\rm loc}:\R^n\tto\R^n$  be given by $\gph F_{\rm loc}=\gph F\cap (X\times Y)$, where $(X\times Y)$ is as in  Definition \ref{DefLocMon}. Then $F_{\rm loc}$ is monotone, has a maximally monotone extension $\tilde F$, cf. \cite[Proposition 12.6]{RoWe98} and $\gph F\cap (X\times Y)=\gph \tilde F\cap (X\times Y)$. By \cite[Corollary 23.8]{BauCom11}, $\rge(I+\tilde F)=\R^n$ and the resolvent $f:=(I+\tilde F)^{-1}$ is a single-valued, firmly nonexpansive mapping on $\R^n$. Since $f(\xb+\yb)=\xb$,  we can find an open neighborhood $U$ of $\xb+\yb$ such that $f(u)\in X$ and $u-f(u)\in Y$ for all $u\in U$. It follows that $(I+F)^{-1}(u)=(I+\tilde F)^{-1}(u)$ and therefore $I+F$ is strongly metrically regular around $(\xb,\xb+\yb)$. Thus  $F$ is graphically Lipschitzian at $(\xb,\yb)$ and has the \SCD property around $(\xb,\yb)$ by Corollary \ref{CorStrongMetrReg}. Now consider a subspace $L\in\Sp F(x,y)$, where $(x,y)$ is close to $(\xb,\yb)$. By \eqref{EqSpStrMetrReg1} there is a matrix $B\in\overline{\nabla} f(x+y)$ such that $L=\rge(B, I-B)$. Since $f$ is firmly nonexpansive, for every $(u',f(u'))\in \OO_f$ we have $\skalp{\nabla f(u')v,v}\geq \norm{\nabla f(u')v}^2$ and  $\norm{Bv}^2\leq \skalp{Bv,v}$, $v\in\R^n$ follows. This completes the proof.
 \end{proof}
 If $F$ is only locally maximally hypomonotone at $(\xb,\yb)$, it follows that $(1+\gamma) I+F$ is strongly metrically regular around $(\xb,(1+\gamma) \xb+\yb)$ for some $\gamma\geq 0$.
 Hence we obtain the following corollary.
 \begin{corollary}\label{CorHypomonotone}
   Let $F:\R^n\tto\R^n$ be locally maximally hypomonotone at $(\xb,\yb)$. Then there is some $\gamma\geq 0$ such that $F$ is graphically Lipschitzian at $(\xb,\yb)$ of dimension $n$ with transformation mapping $\Phi(x,y)=((1+\gamma)x+y,x)$ and therefore $F$ has the \SCD property  around $(\xb,\yb)$. For every $(x,y)\in\gph F$ sufficiently close to $(\xb,\yb)$  and  every subspace $L\in \Sp F(x,y)$ there is a firmly nonexpansive $n\times n$ matrix $B$ such that $L=\rge(B, I-(1+\gamma)B)$ and $L^*=\rge(B^T, I-(1+\gamma)B^T)$.
 \end{corollary}
 \begin{corollary}
   A mapping $F: \R^n\tto\R^n$ which is locally maximally hypomonotone on a dense subset of its graph is an \SCD mapping.
 \end{corollary}

 We now consider the subdifferential mapping $\partial q$ of some lsc function $q:\R^n\to\bar\R$.
\begin{definition}\label{DefProxReg}
\begin{enumerate}
\item  A function $q:\R^n\to \bar \R$  is {\em prox-regular} at $\xb\in\dom q$ for $\xba\in\partial q(\xb)$ if $q$ is locally lsc around $\xb$ and there
exist $\epsilon> 0$ and $\rho\geq 0$ such that for all $x',x\in \B_{\epsilon}(\xb)$ with $\vert q(x)-q(\xb)\vert\leq \epsilon$ one has
\[q(x')\geq q(x)+\skalp{x^*,x'-x}-\frac \rho2 \norm{x'-x}^2\quad\mbox{whenever}\quad x^*\in \partial q (x)\cap \B_\epsilon(\xba).\]
When this holds for all $\xba \in \partial q(\xb)$, $q$ is said to be prox-regular at $\xb$.
\item A function $q:\R^n\to \bar \R$ is called {\em subdifferentially continuous} at $\xb\in\dom q$ for $\xba\in\partial q(\xb)$ if for any sequence $(x_k,x_k^*)\longsetto{{\gph \partial q}}(\xb,\xba)$ we have $\lim_{k\to\infty}q(x_k)=q(\xb)$. When this holds for all $\xba \in \partial q(\xb)$, $q$ is said to be subdifferentially continuous  at $\xb$.
    \end{enumerate}
\end{definition}

\begin{proposition}\label{PropProxRegularQ}
      Suppose that $q:\R^n\to\bar\R$ is prox-regular and subdifferentially continuous at $\xb$ for $\xb^*\in\partial q(\xb)$. Then $\partial q$ is locally maximally hypomonotone at $(\xb,\xba)$ and there is some $\lambda>0$ such that $\partial q$ is graphically Lipschitzian at $(\xb,\yb)$ with transformation mapping $\Phi(x,x^*)=(x+\lambda x^*,x)$. Thus  $\partial q$ has the \SCD property around $(\xb,\xba)$.  Further, for every $(x,x^*)\in\gph\partial q$ sufficiently close to $(\xb,\xba)$ one has $\Sp^*\partial q(x,x^*)=\Sp\partial q(x,x^*)$ and  for every $L\in\Sp\partial q(x,x^*)$  there is a symmetric positive semidefinite $n\times n$ matrix $B$ such that $L=L^*=\rge(B,\frac 1\lambda(I- B))$.
\end{proposition}
\begin{proof}
 Let $\tilde q(x):=q(x)-\skalp{\xba,x}$. Then $\partial \tilde q(\cdot) =\partial q(\cdot)-\xba$ and $\tilde q$ is prox-regular at $\xb$ for $0$.  By the definition of prox-regularity we have
\[\tilde q(x')\geq \tilde q(\xb)-\frac\rho2\norm{x'-\xb}^2\ \forall x'\in \B_\epsilon(\xb)\]
for some $\epsilon>0$ and some $\rho\geq 0$. Hence the function $\hat q:=\tilde q +\delta_{\B_\epsilon}$ fulfills the baseline assumption of \cite[Section 4]{PolRo96}. By \cite[Proposition 4.8]{PolRo96} the subdifferential mapping $\partial \hat q$ is locally maximally hypomonotone around $(\xb,0)$ and, by the proof of \cite[Theorem 4.7]{PolRo96},  for any $\lambda\in(0,\frac1\rho)$ the mapping $\partial \hat q$ is graphically Lipschitzian with transformation mapping $\Phi(x,x^*)=(x+\lambda  x^*,x)$. Moreover, by \cite[Theorem 4.4]{PolRo96},  $(I+\lambda\partial \hat q)^{-1}$ is locally monotone at $(\xb,\xb)$ and  there holds
\begin{equation}\label{EqGradMorEnv}\nabla e_\lambda\hat q(u)=\frac 1\lambda\big(I- (I+\lambda\partial \hat q)^{-1}\big)(u)\end{equation}
for all $u$ sufficiently close to $\xb$, where
\[e_\lambda\hat q(y):=\inf_x \{ \frac1{2\lambda}\norm{x-y}^2+\hat q(x)\}\]
denotes the Moreau envelope of $\hat q$.  Consider a pair $(x,\hat x^*)$ close to $(\xb,0)$ and a subspace $L\in \Sp \partial \hat q(x,\hat x^*)$. According to Corollary \ref{CorHypomonotone}, there is a matrix $B\in \overline{\nabla}( I+\lambda\partial \hat q)^{-1}(x+\lambda\hat x^*)$ with $L=\nabla \Phi(x,\hat x^*)^{-1}\rge(I, B)=\rge( B,\frac 1\lambda(I-B))$, where we have taken into account
\[\nabla \Phi(x,\hat x^*)=\left(\begin{matrix}
  I&\lambda I\\I&0\end{matrix}\right),\ \nabla \Phi(x,\hat x^*)^{-1}=\left(\begin{matrix}
  0& I\\\frac1\lambda I&-\frac1\lambda I\end{matrix}\right).\]
Since $(I+\lambda\partial \hat q)^{-1}$ is locally monotone at $(\xb,\xb)$, it follows that $B$ is positive semidefinite. Further, by \eqref{EqGradMorEnv} we have
$\frac1\lambda(I-B)\in\overline{\nabla}(\nabla e_\lambda\hat q)(x+\lambda \hat x^*)$. By \cite[Theorem 13.52]{RoWe98}, $\overline{\nabla}(\nabla e_\lambda\hat q)(x+\lambda \hat x^*)$ consists of symmetric matrices and consequently $B$ is symmetric. Since $L^\perp=\rge(\frac 1\lambda(I-B),-B)$, we obtain $L^*=L$ and $\Sp^*\partial \hat q(x,\hat x^*)=\Sp\partial \hat q(x,\hat x^*)$ follows. Now, by taking into account that $\partial \tilde q$ and $\partial \hat q$ coincide near $\xb$ and $\partial q$ differs from $\partial \tilde q$ only by the constant $\xba$, it follows that all the shown properties do not hold only for $\partial \hat q$ but also for $\partial q$.
\end{proof}
\begin{corollary}
  For every lsc function $q:\R^n\to\bar\R$ which is prox-regular and subdifferentially continuous at $x$ for $x^*$ on a dense subset of $\gph\partial q$, its subdifferential mapping $\partial q$ is an \SCD mapping.
\end{corollary}
Clearly, every lsc convex function is prox-regular and subdifferentially continuous at all points of its domain. Further, the proof of Proposition \ref{PropProxRegularQ} holds true with $\lambda=1$ and $(I+\partial q)^{-1}$ is firmly nonexpansive and therefore $\norm{B}\leq 1$ $\forall B\in \overline{\nabla}(I+\partial\hat q)^{-1}(\xb)$. Thus we obtain the following corollary.

\begin{corollary}\label{CorSCDConv}
  For every lsc proper convex function $q:\R^n\to\bar\R$ the subdifferential mapping $\partial q$ is  graphically Lipschitzian of dimension $n$ at every point $(x,x^*)$ of its graph. Hence $\partial q$ is an \SCD mapping and for every $(x,x^*)\in \gph\partial q$ and every $L\in \Sp^*\partial q(x,x^*)=\Sp\partial q(x,x^*)$ there is a symmetric positive semidefinite $n\times n$ matrix $B$ with $\norm{B}\leq 1$ such that $L=\rge(B,I-B)$.
\end{corollary}

\begin{example}\label{ExPoly}
  Let $C\subset\R^n$ be a convex polyhedral set and consider $q=\delta_C$ so that $\partial q= N_C$. By the well-known reduction Lemma, see, e.g., \cite[Lemma 2E.4]{DoRo14}, we have
  $T_{\gph N_C}(x,x^*)=\gph N_{\K_C(x,x^*)}$, where $\K_C(x,x^*):=T_C(x)\cap [x^*]^\perp$ denotes the {\em critical cone} to $C$ at $x$ for $x^*$. Thus $N_C$ is graphically smooth of dimension $n$ at $(x,x^*)$ if and only if $\K_C(x,x^*)$ is a subspace and in this case we have $\widehat\Sp N_C(x,x^*)=K_C(x,x^*)\times K_C(x,x^*)^\perp =\widehat\Sp^* N_C(x,x^*)$. Given $(\xb,\xba)\in\gph N_C$, by \cite[Lemma 4H.2]{DoRo14}, for every sufficiently small neighborhood $W$ of $(\xb,\xba)$, the collection of all critical cones $\K_C(x,x^*)$, $(x,x^*)\in W$ coincides with the collection of all sets of the form $\F_1-\F_2$, where $\F_1,\F_2$ are faces of $\K_C(\xb,\xba)$ with $\F_2\subseteq \F_1$. Since $\F_1-\F_2$ is a subspace if and only if $\F_1=\F_2$ and $\K_C(\xb,\xba)$ has only finitely many faces, we obtain
  \begin{equation}\label{EqSCDNormalCone}\Sp N_C(\xb,\xba)=\Sp^*N_C(\xb,\xba)=\{(\F-\F)\times(\F-\F)^\perp\mv \mbox{$\F$ is face of $\K_C(\xb,\xba)$}\}.\end{equation}
  Of course, for every face $\F$ of $\K_C(\xb,\xba)$ we have
  \[(\F-\F)\times(\F-\F)^\perp=\rge(B,I-B),\]
  where $B$ represents the orthogonal projection on $\F-\F$.

Let us compare the representation \eqref{EqSCDNormalCone} with the limiting coderivative $D^* N_C(\xb,\xba)$.  It was shown in \cite[Proof of Theorem 2]{DoRo96} that
 $N_{\gph N_C}(\xb,\xba)$ is the union of all product sets $K^\circ\times K$ associated with cones $K$ of the form $\F_1-\F_2$, where $\F_1$ and $\F_2$ are closed faces of the critical cone $K_C(\xb,\xba)$ satisfying $\F_2\subset \F_1$. Thus, $\gph D^*N_C(\xb,\xba)$ is the union of all respective sets of the form $(\F_2-\F_1)\times(\F_1-\F_2)^\circ$ and we see that $\Sp^*N_C(\xb,\xba)$ has a simpler structure than the limiting coderivative $D^* N_C(\xb,\xba)$ whenever the critical cone $K_C(\xb,\xba)$ is not a subspace.
\end{example}
Given a function $q:\R^n\to\bar\R$, $\Sp^*\partial q(x,x^*)$ amounts to a generalized derivative of the subgradient mapping and  constitutes therefore some generalized second-order derivative of $q$. In the framework of a study of sufficient conditions for local optimality, Rockafellar \cite{Ro21} has introduced another type of generalized second-order derivative as an epigraphical limit of certain second-order subderivatives.
 \begin{definition}[\cite{Ro21}]
   \begin{enumerate}
     \item A function $\phi:\R^n\to\bar \R$ is called a {\em generalized quadratic form}, if $\phi(0)=0$ and the subgradient mapping $\partial \phi$ is {\em generalized linear}, i.e., $\gph \partial \phi$ is a subspace of $\R^n\times\R^n$.
     \item A function $q:\R^n\to\bar \R$ is called {\em generalized twice differentiable} at $x$ for a subgradient $x^*\in\partial q(x)$, if it is twice epi-differentiable at $x$ for $x^*$ with the second-order subderivative ${\rm d}^2q(x, x^*)$ being a generalized quadratic form.
     \item Given a function $q:\R^n\to\bar \R$ and a pair $(x,x^*)\in\gph \partial q$, the {\em quadratic bundle} of $q$ at $x$ for $x^*$ is defined by
     \[{\rm quad\,}q(x,x^*):=\left[\ \begin{minipage}{10cm}the collection of generalized quadratic forms $\phi$ for which $\exists (x_k,x_k^*)\to(x,x^*)$ with $q$
             generalized twice differentiable at $x_k$ for $x_k^*$ and such that the generalized quadratic forms $\phi_k={\rm d}^2q(x_k, x_k^*)$ converge epigraphically to $\phi$.
     \end{minipage}\right.\]
   \end{enumerate}
 \end{definition}
 We establish now a strong relationship between $\Sp\partial q$ and ${\rm quad\,}q$ for prox-regular and subdifferentially continuous functions $q$. We start with the following lemma.
 \begin{lemma}\label{LemGenTwDiff}Let $q:\R^n\to\bar \R$ be prox-regular and subdifferentially continuous at $x$ for $x^*\in\partial q(x)$. Then $(x,x^*)\in\OO_{\partial q}$ if and only if $q$ is generalized twice differentiable at $x$ for $x^*$, and in this case one has
   \begin{equation}\label{EqSCDTwEpi}\widehat \Sp\partial q(x,x^*)=\{\gph \partial\big(\frac 12{\rm d^2}q(x, x^*)\big)\}.\end{equation}
 \end{lemma}
 \begin{proof}
 By \cite[Theorem 13.40]{RoWe98}, the subgradient mapping $\partial q$ is proto-differentiable at $(x,x^*)$ if and only if $q$ is twice epi-differentiable at $x$ for $x^*$, and then
  \[D(\partial q)(x,x^*)=\partial\big(\frac 12{\rm d^2}q(x, x^*)\big).\]
  By Proposition \ref{PropProxRegularQ}, $\partial q$ is graphically Lipschitzian of order $n$ at $(x,x^*)$. Hence, if $(x,x^*)\in\OO_{\partial q}$ then $\partial q$ is proto-differentiable at $(x,x^*)$ by Remark \ref{RemProtoDiff} and consequently $q$ is twice epi-differentiable at $x$ for $x^*$ and $\gph \partial\big(\frac 12{\rm d^2}q(x, x^*)\big)=\gph D(\partial q)(x,x^*)=T_{\gph \partial q}(x,x^*)$ is a subspace. This verifies that $q$ is generalized twice differentiable at $x$ for $x^*$. Conversely, if $q$ is generalized twice differentiable at $x$ for $x^*$, then $T_{\gph \partial q}(x,x^*)=\gph \partial\big(\frac 12{\rm d^2}q(x, x^*)\big)$ is a subspace.  The dimension of this subspace must be $n$ because $\partial q$ is graphically Lipschitzian of dimension $n$, and $(x,x^*)\in\OO_{\partial q}$ follows.
\end{proof}
We will also make use of the following variant of Attouch's theorem, see, e.g., \cite[Theorem 12.35]{RoWe98}, which states the connection between graphical convergence of
subdifferential mappings and epi-convergence of the functions themselves, when the functions are convex.
\begin{lemma}\label{LemAttouch}
  Let $\phi_k:\R^n\to\bar\R$, $k\in\N$ and $\phi:\R^n\to\bar\R$ be proper lsc functions and assume that there is some $\rho\geq 0$ such that the functions $\hat\phi_k:=\phi_k+\rho\norm{\cdot}^2$ are convex. Then the following two statements are equivalent.
  \begin{enumerate}
    \item[(i)]The functions $\phi_k$ epi-converge to $\phi$.
    \item[(ii)]The mappings $\partial \phi_k$ converge graphically to $\partial\phi$, $\hat\phi:=\phi+\rho\norm{\cdot}^2$ is convex and there is some sequence $(x_k,x_k^*)$ converging to some $(\xb,\xba)\in\gph\partial \phi$ such that $(x_k,x_k^*)\in\gph\partial\phi_k$ $\forall k$ and $\lim_{k\to\infty}\phi_k(x_k)=\phi(\xb)$.
  \end{enumerate}
\end{lemma}
\begin{proof}Observe that assertion (i) is equivalent to epi-convergence of $\hat\phi_k\to\hat\phi$ by \cite[Exercise 7.8]{RoWe98}, and in this case the epigraphical limit $\hat\phi$ is convex, cf. \cite[Theorem 7.17]{RoWe98}. Similarly, $\partial\phi_k$ converges graphically to $\partial\phi$ if and only if $\partial\hat\phi_k$ converges graphically to $\hat\phi$. Indeed, since $\partial \hat\phi_k(x)=\partial\phi_k(x)+2\rho x$, we easily obtain
\begin{align*}&\Limsup_{k\to\infty}\gph\hat\phi_k=\{(x,x^*+2\rho x)\mv (x,x^*)\in \Limsup_{k\to\infty}\gph\phi_k\},\\
 &\Liminf_{k\to\infty}\gph\hat\phi_k=\{(x,x^*+2\rho x)\mv (x,x^*)\in \Liminf_{k\to\infty}\gph\phi_k\}
 \end{align*}
 and the claim follows. Finally we have $\phi(\xb)=\lim_{k\to\infty} \phi_k(x_k)$ for some sequence $(x_k,x_k^*)\to(\xb,\xba)\in \gph\partial \phi$ with $(x_k,x_k^*)\in\gph\partial \phi_k$, if and only if $\hat\phi(\xb)=\lim_{k\to\infty} \hat\phi_k(x_k)$ and $\gph\partial \hat\phi_k\ni(x_k,x_k^*+2\rho x_k)\to(\xb,\xba+2\rho\xb)\in \gph\partial \hat\phi$. Now the equivalence  between (i) and (ii) follows from Attouch's theorem \cite[Theorem 12.35]{RoWe98} applied to the convex functions $\hat\phi_k$ and $\hat\phi$.
\end{proof}
\begin{proposition}Suppose that $q:\R^n\to\bar\R$ is prox-regular at $\xb$ for $\xb^*\in\partial q(\xb)$. Further assume that for all $(x,x^*)\in\gph \partial q$ sufficiently close to $(\xb,\xba)$ the function $q$ is subdifferentially continuous at $x$ for $x^*$. Then
\[\Sp^*\partial q(\xb,\xba)=\Sp\partial q(\xb,\xba)=\{\gph \partial(\frac 12 \phi)\mv \phi\in {\rm quad\,}q(\xb,\xba)\}.\]
\end{proposition}
\begin{proof}
  Let $\epsilon$ and $\rho\geq 0$ be as in Definition \ref{DefProxReg}. Then we can find an open neighborhood $W$ of $(\xb,\xba)$ such that $W\subset\B_{\frac\epsilon2}(\xb)\times\B_{\frac \epsilon2}(\xba)$ and for all $(x,x^*)\in \gph\partial q\cap W$ the function  $q$ is subdifferentially continuous at $x$ for $x^*$ and  $\vert q(x)-q(\xb)\vert<\frac\epsilon2$.
  Consider $(\hat x,\hat x^*)\in\gph\partial q\cap W$. Then for all $x,x'\in\B_{\frac \epsilon2}(\hat x)$ with $\vert q(x)-q(\hat x)\vert\leq\frac\epsilon2$ and all $x^*\in \partial q(x)\cap \B_{\frac\epsilon2}(\hat x^*)$ we have $x,x'\in \B_\epsilon(\xb)$, $\vert q(x)-q(\xb)\vert\leq \epsilon$ and $x^*\in\partial q(x)\cap\B_\epsilon(\xba)$ implying
  $q(x')\geq q(x)+\skalp{x^*,x'-x}-\frac\rho2\norm{x'-x}^2$. Thus $q$ is prox-regular at $\hat x$ for $\hat x^*$ and therefore graphically Lipschitzian of dimension $n$ at $(\hat x,\hat x^*)$. Further, by \cite[Proposition 13.49]{RoWe98} and its proof, we may conclude that ${\rm d^2}q(\hat x,\hat x^*)+\rho\norm{\cdot}^2$ is a lsc convex function.\\
  Now consider $\phi\in{\rm quad\,}q(\xb,\xba)$ together with some sequence $(x_k,x_k^*)\to(\xb,\xba)$ such that $q$ is generalized twice differentiable at $x_k$ for $x_k^*$ and the generalized quadratic forms $\phi_k:={\rm d}^2q(x_k, x_k^*)$ converge epigraphically to $\phi$. We may assume that $(x_k,x_k^*)\in\gph \partial q\cap W$ $\forall k$ and thus $(x_k,x_k^*)\in\OO_{\partial q}\cap W$ by Lemma \ref{LemGenTwDiff}. By Lemma \ref{LemAttouch}, $\partial \phi_k$ converges graphically to $\partial\phi$, i.e., $\gph \partial\phi_k\to\gph\partial\phi$ in the sense of Painlev\'e-Kuratowski convergence. Now it follows from Lemma \ref{LemCompact}(v) and \eqref{EqSCDTwEpi} that $\gph \partial(\frac 12 \phi)\in \Sp\partial q(\xb,\xba)$.\\
  Conversely, consider a subspace $L\in\Sp\partial q(\xb,\xba)$ together with sequences $(x_k,x_k^*)\in\OO_{\partial q}\cap W$ and $L_k\in\widehat\Sp\partial q(x_k,x_k^*)$ with $\lim_{k\to\infty}d_\Z(L_k,L)=0$. According to Proposition \ref{PropProxRegularQ} we can find symmetric positive semidefinte matrices $B_k$, $B$ such that $L_k=\rge(B_k,\frac 1\lambda(I-B_k))$, $L=\rge(B,\frac 1\lambda(I-B))$ with $\lambda=1/(\rho+1)$. Now let $\U:=\rge B$ and set $Q:=\frac1\lambda(B^\dag-BB^\dag)$, where $B^\dag$ denotes the Moore-Penrose inverse of $B$. Since $B$ is symmetric and positive semidefinite, so is $B^\dag$ as well. Further, $BB^\dag B=B$ and $BB^\dag=B^\dag B$ is the orthogonal projection onto $\rge B$, so that $I-B^\dag B$ is the orthogonal projection onto $\U^\perp=\ker B$. Now consider the generalized quadratic form $\phi(x):=\skalp{Qx,x}+\delta_\U(x)$. Since $Q$ is symmetric, we obtain
  \[\partial (\frac 12\phi)(x)=\begin{cases}Qx+\U^\perp=\frac1\lambda(B^\dag-BB^\dag)x+\U^\perp&\mbox{if $x\in \U$,}\\\emptyset&\mbox{else.}\end{cases}\]
  Thus
  \begin{align*}\gph \partial(\frac 12\phi)&=\Big\{\big(Bp,\frac1\lambda(B^\dag-BB^\dag)Bp+(I-B^\dag B)v\big)\mv p,v\in\R^n\Big\}\\
  &=\Big\{\Big(B\big(B^\dag B p+(I-B^\dag B)v\big),\frac 1\lambda(I-B)\big(B^\dag Bp+(I-B^\dag B)v\big)\Big)\mv p,v\in\R^n\Big\}\\
  &=\rge(B,\frac1\lambda(I-B))=L.\end{align*}
  The matrix $Q+\frac1\lambda I=\frac 1\lambda B^\dag+\frac1\lambda(I-BB^\dag)$ is positive semidefinite as the sum of two positive semidefinite matrices and therefore the function $\hat\phi:=\phi+\frac1\lambda\norm{\cdot}^2=\phi+(\rho+1)\norm{\cdot}^2$ is convex. For each $k$, the function $\phi_k:={\rm d}^2q(x_k,x_k^*)$ fulfills $\gph \partial(\frac 12q)=L_k$ by Lemma \ref{LemGenTwDiff} and $\hat \phi_k:=\phi_k+(\rho+1)\norm{\cdot}^2$ is convex. Since $(0,0)\in L_k=\gph\partial (\frac 12\phi_k)$, we  have $0\in \partial \phi_k(0)$. Further, $\phi_k(0)=\phi(0)=0$ and $0\in \partial \phi(0)$. Since convergence of $L_k$ to $L$ implies that $\partial \phi_k$ converges graphically to $\partial \phi$, it follows from Lemma \ref{LemAttouch}  that $\phi_k$ converges epigraphically to $\phi$ and we conclude $\phi\in {\rm quad\,}q(\xb,\xba)$. Thus $L\in\{\gph\partial(\frac12 \phi)\mv \phi\in{\rm quad\,}q(\xb,\xba)\}$ verifying $\Sp\partial q(\xb,\xba)=\{\gph\partial(\frac12 \phi)\mv \phi\in{\rm quad\,}q(\xb,\xba)\}$. By Proposition \ref{PropProxRegularQ} we have $\Sp^*\partial q(\xb,\xba)=\Sp\partial q(\xb,\xba)$ and the proof is complete.
\end{proof}
\begin{corollary}For every lsc proper convex function $q:\R^n\to\bar\R$  and for every pair $(x,x^*)\in \gph\partial q$ we have
  \[\Sp^*\partial q(x,x^*)=\Sp\partial q(x,x^*)=\{\gph \partial(\frac 12 \phi)\mv \phi\in {\rm quad\,}q(x,x^*)\}.\]
\end{corollary}
\section{\SCD regularity}
In this section we present the definition and basic properties of a certain property called \SCD regularity, which has various applications as we will demonstrate in the subsequent sections.
\begin{definition}\begin{enumerate}
\item We denote by $\Z_n^{\rm reg}$ the collection of all subspaces $L\in\Z_n$ such that
  \begin{equation}\label{EqSCDReg_L}
    (y^*,0)\in L\ \Rightarrow\ y^*=0.
  \end{equation}
  \item  A mapping $F:\R^n\tto\R^n$ is called {\em \SCD regular around} $(x,y)\in\gph F$, if $F$ has the \SCD property around $(x,y)$ and
  \begin{equation}\label{EqSCDReg}
    (y^*,0)\in \bigcup \Sp^*F(x,y) \Rightarrow\ y^*=0,
  \end{equation}
  i.e., $L\in \Z_n^{\rm reg}$ for all $L\in \Sp^*F(x,y)$. Further, we will denote by
  \[{\rm scd\,reg\;}F(x,y):=\sup\{\norm{y^*}\mv (y^*,x^*)\in \bigcup \Sp^*F(x,y), \norm{x^*}\leq 1\}\]
  the {\em modulus of \SCD regularity} of $F$ around $(x,y)$.
\end{enumerate}
\end{definition}
In the following proposition we state some basic properties of subspaces $L\in\Z_n^{\rm reg}$.
\begin{proposition}\label{PropC_L}
    Given a $2n\times n$ matrix $Z$, there holds $\rge Z\in \Z_n^{\rm reg}$ if and only if the $n\times n$ matrix $\pi_2(Z)$ is nonsingular. Thus,
    for every $L\in \Z_n^{\rm reg}$   there is a  unique $n\times n$ matrix $C_L$  such that $L=\rge(C_L,I)$. Further, $L^*=\rge(C_L^T,I)\in\Z_n^{\rm reg}$,
    \begin{equation}\label{EqC_L}\skalp{x^*,C_L^Tv}=\skalp{y^*,v}\ \forall (y^*,x^*)\in L\ \forall v\in\R^n.
    \end{equation}
    and
    \begin{equation}\label{EqKappa_L}
    \norm{y^*}\leq \norm{C_L}\norm{x^*}\ \forall (y^*,x^*)\in L.
  \end{equation}
\end{proposition}
\begin{proof}
Clearly, if $\pi_2(Z)$ is nonsingular then $\rge Z\in \Z_n$. Further, given $(y^*,0)\in\rge Z$, there is some $p$ with $y^*=\pi_1(Z)p$, $0=\pi_2(Z)p$ implying $p=y^*=0$ and therefore $\rge Z\in\Z_n^{\rm reg}$. Conversely, consider $L
\in\Z_n^{\rm reg}$ and $Z\in \M(L)$. Because $L\in \Z_n$, the matrix $Z$ has full column rank $n$ and therefore there cannot exist $p\not=0$ with $Zp=0$.  Thus, if  $A:=\pi_2(Z)$ were singular, there is some $0\not=p\in\R^n$ with $Ap=0$ and $Bp\not=0$ with $B:=\pi_1(Z)$, implying $(Bp,0)\in L$ and $Bp\not=0$ which is not possible because of $L\in\Z_n^{\rm reg}$. This proves that $A$ is nonsingular and $L=\rge ZA^{-1}= \rge (C_L,I)$ with $C_L=BA^{-1}$ follows.
 Clearly, $C_L$ is uniquely given by $L$ and does not depend on the particular choice of $A$ and $B$.
From $L=\rge(C_L,I)$ we deduce $L^\perp=\rge(I,-C_L^T)$ and $L^*=S_nL^\perp=\rge(C_L^T,I)$. Further, for every $p\in\R^n$ we have $(p,-C_L^Tp)\in L^\perp$, implying
\[\skalp{p,y^*}-\skalp{C_L^Tp,x^*}=\skalp{p,y^*-C_Lx^*}=0\ \forall (y^*,x^*)\in L\]
and \eqref{EqC_L} follows. Finally, for every $(y^*,x^*)\in L$ there is some $p\in\R^n$ with $y^*=C_Lp$, $x^*=p$ implying \eqref{EqKappa_L}.
\end{proof}
\begin{remark}Note that for every $L\in\Z_n^{\rm reg}$  there holds $C_L=\pi_1(Z)\pi_2(Z)^{-1}$ for every $Z\in \M(L)$.
\end{remark}
In case of \SCD regularity we obtain from Proposition \ref{PropC_L} that
\[\bigcup\Sp^*F(x,y)=\{(C_Lp,p)\mv L\in \Sp^*F(x,y), p\in\R^n\}\]
and consequently
\begin{align}
  \nonumber{\rm scd\,reg\;}F(x,y)&=\sup\{\norm{C_Lp}\mv L\in \Sp^*F(x,y), p\in\R^n,\norm{p}\leq 1\}\\
  \label{EqSCDRegMod}&=\sup\{\norm{C_L}\mv L\in \Sp^*F(x,y)\}.
\end{align}
\begin{remark}\label{RemSCDRegSingleValued}
  In case of a single-valued, locally Lipschitzian mapping $F:\R^n\to\R^n$, by virtue of Lemma \ref{LemSCDSingleValued}, \SCD regularity of $F$ around $\xb$ means that all matrices  belonging to the B-subdifferential are nonsingular. This is exactly the so-called BD-regularity property from \cite{Qi93}.
\end{remark}
By isometry of the mapping $L\mapsto L^*$ and Proposition \ref{PropC_L} we obtain readily the following lemma.
\begin{lemma}\label{LemSCCRegPrimal}
  The mapping $F:\R^n\tto\R^n$ is \SCD regular around $(\xb,\yb)\in\gph F$ if and only if
  \begin{equation}\label{EqSCDReg1}
    (u,0)\in \bigcup\Sp F(x,y)\ \Rightarrow\ u=0.
  \end{equation}
  Further,
  \[{\rm scd\,reg\;}F(x,y)=\sup\{\norm{u}\mv (u,v)\in \bigcup\Sp F(x,y),\ \norm{v}\leq 1\}=\sup\{\norm{C_L}\mv L\in \Sp F(x,y)\}.\]
\end{lemma}
 Note that \SCD regularity is weaker than the metric regularity of $F$ around $(x,y)$. Indeed,
condition \eqref{EqMoCrit} for metric regularity of $F$ near $(x,y)$ can be equivalently written as
\[(y^*,0)\in \gph D^*F(x,y)\ \Rightarrow\ y^*=0\]
and  $\bigcup\Sp^*F(x,y)$ is contained in $\gph D^*F(x,y)$ by Lemma \ref{LemSCDname}.
The next example shows that \SCD regularity is even strictly weaker than metric regularity, see also Example \ref{ExSubspReg} below.
\begin{example}Consider the \SCD mapping
\[ F(x):=-x+ N_{R_-}(x) =\partial q(x)\ \mbox{ with } q(x)=-\frac 12 x^2+\delta_{R_-}(x)\]
at $(0,0)$. Then
\[D^*F(0,0)(y^*)=-y^*+\begin{cases}\{0\}&\mbox{if $y^*<0$}\\
\R&\mbox{if $y^*=0$}\\
\R_+&\mbox{if $y^*>0$}\end{cases}\]
and therefore the only subspaces contained in $\gph D^*F(0,0)$ are $\{(y^*,-y^*)\mv y^*\in\R\}$ and $\{0\}\times\R$. Hence $F$ is \SCD regular at $(0,0)$, but $F$ is not metrically regular near $(0,0)$ because of $0\in D^*F(0,0)(1)$.
\end{example}
\begin{lemma}
  Let $F:\R^n\tto\R^n$ be  \SCD regular around $(x,y)\in \gph F$. Then ${\rm scd\,reg\;}F(x,y)<\infty$.
\end{lemma}
\begin{proof}
 Assume on the contrary that there are  sequences  $L_k\in\Sp^*F(x,y)$ and $(y_k^*,x_k^*)\in L_k$ such that $\norm{y_k^*}\geq k$ and $\norm{x_k^*}\leq1$. By possibly passing to some subsequence we can assume that $y_k^*/\norm{y_k^*}$ converges to some $y^*$ with $\norm{y^*}=1$ and $L_k$ converges in the compact metric space $\Z_n$ to some $L$. Then $(y^*,0)=\lim_{k\to\infty}(y_k^*,x_k^*)/\norm{y_k^*}\in L$ and $L\in\Sp^*F(x,y)$ by Lemma \ref{LemLimSupSp} contradicting the assumption of \SCD regularity.
\end{proof}
\begin{proposition}\label{PropSCDReg}
  Assume that $F:\R^n\tto\R^n$ is \SCD regular around $(\xb,\yb)\in\gph F$. Then $F$ is \SCD regular around every $(x,y)\in\gph F$ sufficiently close to $(\xb,\yb)$ and
  \[\limsup_{(x,y)\longsetto{\gph F}(\xb,\yb)}{\rm scd\,reg\;}F(x,y)\leq{\rm scd\,reg\;}F(\xb,\yb).\]
\end{proposition}
\begin{proof}
  By contraposition. If any of the assertions does not hold, we can find some $\kappa>{\rm scd\,reg\;}F(\xb,\yb)$ and sequences $(x_k,y_k)\longsetto{{\gph F}}(\xb,\yb)$, $L_k\in\Sp^*F(x_k,y_k)$ and $(y_k^*,x_k^*)\in L_k$ with $\norm{y_k^*}\geq\kappa$ and $\norm{x_k^*}\leq 1$. By possibly passing to some subsequence we can assume that $(y_k^*,x_k^*)/\norm{y_k^*}$ converges to some $(y^*,x^*)$ and $L_k$ converges to some $L$. Then $\norm{y^*}=1$, $\norm{x^*}\leq \frac 1{\kappa}$, $(y^*,x^*)\in L$ and $L\in\Sp^*F(\xb,\yb)$ by Lemma \ref{LemLimSupSp}. Since $L$ is a subspace, we also have $(\kappa y^*,\kappa x^*)\in L\subseteq\bigcup\Sp^*F(\xb,\yb)$ implying together with $\norm{\kappa x^*}\leq 1$ the contradiction ${\rm scd\,reg\;}F(\xb,\yb)\geq \norm{\kappa y^*}=\kappa$.
\end{proof}
\section{On semismooth* Newton methods for \SCD mappings\label{SecSemiSm}}
Consider the inclusion
\begin{equation}\label{EqIncl}
  0\in F(x),
\end{equation}
where $F:\R^n\tto\R^n$. Assume that $\xb$ is a reference solution of \eqref{EqIncl}.
The idea behind the \ssstar Newton method \cite{GfrOut21} for solving \eqref{EqIncl} is as follows. If $F$ is \ssstar at $(\xb,0)$ and we are given some point $(x,y)\in\gph F$ close to $(\xb,0)$, then for every $(y^*,x^*)\in\gph D^*F(x,y)$ there holds
\[\skalp{x^*,x-\xb}=\skalp{y^*,y-0}+\oo(\norm{(x,y)-(\xb,0)}\norm{(x^*,y^*)})\]
by the definition of the semismoothness* property. We choose now $n$ pairs $(y_i^*,x_i^*)\in \gph D^*F(x,y)$, $i=1,\ldots, n$, compute a solution $\Delta x$ of the system
\begin{equation}\label{EqBasSystem}\skalp{x_i^*,\Delta x}=-\skalp{y_i^*,y},\ i=1,\ldots,n\end{equation}
and expect that $\norm{(x+\Delta x)-\xb}=\oo(\norm{(x,y)-(\xb,0)}$. When dealing with \SCD mappings $F$ we can simplify this procedure by choosing the pairs $(y_i^*,x_i^*)$ as a basis of some subspace $L\in\Sp^*F(x,y)$, which allows us to weaken the notion of semismoothness* along the lines of Proposition \ref{PropCharSemiSmooth}.

\begin{definition}\label{DefSCDssstar}
We say that $F:\R^n\tto\R^n$ is {\em\SCD \ssstar at} $(\xb,\yb)\in\gph F$ if $F$ has the \SCD property around $(\xb,\yb)$ and
for every $\epsilon>0$ there is some $\delta>0$ such that
\begin{align}\label{EqDefSSCSemiSmooth}
\vert \skalp{x^*,x-\xb}-\skalp{y^*,y-\yb}\vert&\leq \epsilon
\norm{(x,y)-(\xb,\yb)}\norm{(x^*,y^*)}
\end{align}
holds for all $(x,y)\in \gph F\cap \B_\delta(\xb,\yb)$ and all $(y^*,x^*)\in\bigcup\Sp^*F(x,y)$.\\
We say that $F:\R^n\tto\R^n$ is {\em \SCD \ssstar around} $(\xb,\yb)\in\gph F$ if there is some neighborhood $W$ of $(\xb,\yb)$ such that $F$ is \SCD \ssstar at every $(x,y)\in\gph F\cap W$.
\end{definition}

The chosen subspace $L\in \Sp^*F(x,y)$ should have the property that the resulting system \eqref{EqBasSystem} has a unique solution. Taking the pairs $(y_i^*,x_i^*)$, $i=1,\ldots,n$, as columns of a $2n\times n$ matrix $Z$, this yields the requirement that $\pi_2(Z)$ is nonsingular, which in turn is equivalent to $L\in\Z_n^{\rm reg}$ by Proposition \ref{PropC_L}.
Comparing \eqref{EqC_L} with \eqref{EqBasSystem}, we see that $\Delta x=-C_L^Ty$ is a solution to \eqref{EqBasSystem} in this case.

We are now in the position to describe the iteration step of the \SCD variant of the \ssstar Newton method introduced in \cite{GfrOut21}. Assume we are given some iterate $x^{(k)}$. Since we cannot expect in general that $F(x^{(k)})\not=\emptyset$ or that $0$ is close to
$F(x^{(k)})$, even if $x^{(k)}$ is close to a solution $\xb$,  we first perform  some  preparatory step which
yields $(\hat x^{(k)},\hat y^{(k)})\in\gph F$ as  an approximate projection of $(x^{(k)},0)$ onto $\gph F$. We require that
\begin{equation}\label{EqBndApprStep}
\norm{(\hat x^{(k)},\hat y^{(k)})-(\xb,0)}\leq \eta\norm{x^{(k)}-\xb}
\end{equation}
for some constant $\eta>0$. E.g., if
\[\norm{(\hat x^{(k)},\hat y^{(k)})-(x^{(k)},0)}\leq \beta\dist{(x^{(k)},0),\gph F}\]
holds with some $\beta\geq 1$, then
\begin{align*}\norm{(\hat x^{(k)},\hat y^{(k)})-(\xb,0)}&\leq \norm{(\hat x^{(k)},\hat y^{(k)})-(x^{(k)},0)}+\norm{(x^{(k)},0)-(\xb,0)}\\
&\leq  \beta\dist{(x^{(k)},0),\gph F}+\norm{(x^{(k)},0)-(\xb,0)}\leq (\beta+1)\norm{(x^{(k)},0)-(\xb,0)}
\end{align*}
and \eqref{EqBndApprStep} holds with $\eta=\beta+1$. Further we require that $\Sp^*F(\hat x^{(k)},\hat y^{(k)})\cap\Z_n^{\rm reg}\not=\emptyset$ and  compute the new iterate as $x^{(k+1)}=\hat x^{(k)}-C_L^T\hat y^{(k)}$ for some $L\in \Sp^*F(\hat x^{(k)},\hat y^{(k)})\cap\Z_n^{\rm reg}$. In fact, in a numerical implementation we will not calculate the matrix $C_L$, but two $n\times n$ matrices $A,B$ such that $L=\rge(B^T,A^T)$, compute $\Delta x^{(k)}$ as a solution of the system $A\Delta x=-B\hat y^{(k)}$ and set $x^{(k+1)}=\hat x^{(k)}+\Delta x^{(k)}$.

This leads to the following conceptual algorithm.
\begin{algorithm}[\SCD \ssstar Newton-type method for inclusions]\label{AlgNewton}\mbox{ }\\
 1. Choose a starting point $x^{(0)}$, set the iteration counter $k:=0$.\\
 2. If ~ $0\in F(x^{(k)})$, stop the algorithm.\\
  3. \begin{minipage}[t]{\myAlgBox} {\bf Approximation step: } Compute
  $$(\hat x^{(k)},\hat y^{(k)})\in\gph F$$ satisfying \eqref{EqBndApprStep} such that $\Sp^*F(\hat x^{(k)},\hat y^{(k)})\cap\Z_n^{\rm reg}\not=\emptyset$.\end{minipage}\\
  4. \begin{minipage}[t]{\myAlgBox} {\bf Newton step: }Select $n\times n$ matrices $A^{(k)},B^{(k)}$ with $L^{(k)}:=\rge\big({B^{(k)}}^T,{A^{(k)}}^T)\in \Sp^*F(\hat x^{(k)},\hat y^{(k)})\cap\Z_n^{\rm reg}$, calculate the Newton direction $\Delta x^{(k)}$ as a solution of the linear system $A^{(k)}\Delta x=-B^{(k)}\hat y^{(k)}$ and obtain the new iterate via $x^{(k+1)}=\hat x^{(k)}+\Delta x^{(k)}.$\end{minipage}\\
  5. Set $k:=k+1$ and go to 2.
\end{algorithm}
  We have $\Delta x^{(k)}=-C_{L^{(k)}}^T\hat y^{(k)}$ and therefore $(\Delta x^{(k)},-\hat y^{(k)})\in -{L^{(k)}}^*={L^{(k)}}^* \in \Sp F(\hat x^{(k)},\hat y^{(k)})$ by Proposition \ref{PropC_L}. Thus, alternatively we can perform the Newton step also in the following way:\\
  4. \begin{minipage}[t]{\myAlgBox} {\bf Newton step: }Select $n\times n$ matrices $A^{(k)},B^{(k)}$ with $\rge\big({B^{(k)}},{A^{(k)}})\in \Sp F(\hat x^{(k)},\hat y^{(k)})\cap\Z_n^{\rm reg}$, compute a solution $p$ of the linear system ${A^{(k)}}p =-\hat y^{(k)}$  and compute the new iterate $x^{(k+1)}=\hat x^{(k)}+\Delta x^{(k)}$ with Newton direction $\Delta x^{(k)}=B^{(k)}p$.\end{minipage}\\
\begin{remark}  Note that  $-\hat y^{(k)}\in D^\sharp F(\hat x^{(k)},\hat y^{(k)})(\Delta x^{(k)})$ but we do not necessarily have  $-\hat y^{(k)}\in  DF(\hat x^{(k)},\hat y^{(k)})(\Delta x^{(k)})$ as  it is the case in Newton methods based on the graphical derivative, cf. \cite{Dias,HKM,MoSa21}.
\end{remark}
  Which possibility for calculating the Newton direction is actually chosen, depends on the availability of the respective derivative. Let us analyze the two alternatives for the special case when $F$ is single-valued and continuously differentiable at $x^{(k)}$. In Algorithm \ref{AlgNewton}, the matrices $A^{(k)},B^{(k)}$ with $\rge\big({B^{(k)}}^T,{A^{(k)}}^T)\in \Sp^*F(\hat x^{(k)},\hat y^{(k)})$ fulfill ${A^{(k)}}^T=\nabla F(\hat x^{(k)})^T {B^{(k)}}^T$ and thus the Newton direction is computed by solving the linear system $\big(B^{(k)}\nabla F(\hat x^{(k)})\big)\Delta x=-B^{(k)}\hat y^{(k)}$. On the other hand, given $A^{(k)},B^{(k)}$ with $\rge\big({B^{(k)}},{A^{(k)}})\in \Sp F(\hat x^{(k)},\hat y^{(k)})$, we have $A^{(k)}=\nabla F(\hat x^{(k)})B^{(k)}$ and in this case the Newton direction is computed via $\Delta x^{(k)}=B^{(k)}p=-B^{(k)}\big(\nabla F(\hat x^{(k)})B^{(k)}\big)^{-1}\hat y^{(k)}$. The structure of the second approach resembles the adjoint system method known from PDE-constrained optimization and optimal control.

We now consider convergence of Algorithm \ref{AlgNewton}.
\begin{proposition}\label{PropConvNewton}
  Assume that $F:\R^n\tto\R^n$ is \SCD \ssstar at $(\xb,\yb)\in\gph F$. Then for every  $\epsilon>0$ there is  some $\delta>0$ such that the inequality
  \begin{equation}\label{EqBndNewtonStep}\norm{x-C_L^T(y-\yb)-\xb}\leq \epsilon\sqrt{n(1+\norm{C_L}^2)}\norm{(x,y)-(\xb,\yb)}\end{equation}
  holds for every $(x,y)\in\gph F\cap \B_\delta(\xb,\yb)$ and every $L\in\Sp^*F(x,y)\cap\Z_n^{\rm reg}$.
\end{proposition}
\begin{proof}
Pick $\epsilon>0$ and choose $\delta>0$ such that \eqref{EqDefSSCSemiSmooth} holds. Now consider any
  $(x,y)\in \gph F\cap \B_\delta(\xb,\yb)$ and any $L\in\Sp^*F(x,y)\cap\Z_n^{\rm reg}$. By Proposition \ref{PropC_L} we have $L=\rge(C_L,I)$ and therefore $(C_Le_i,e_i)\in L$, $i=1,\ldots,n$, where $e_i$ denotes the $i$-th unit vector. From \eqref{EqDefSSCSemiSmooth} we obtain
  \begin{align*}\vert\skalp{e_i, \xb-x}-\skalp{C_L e_i,\yb-y}\vert &=\vert\skalp{e_i,x-C_L^T(y-\yb)-\xb}\vert\leq \epsilon \norm{(e_i, C_Le_i)}\norm{(x,y)-(\xb,\yb)}\\
  &\leq \epsilon\sqrt{1+\norm{C_L}^2}\norm{(x,y)-(\xb,\yb)}\end{align*}
  and
  \[\norm{x-C_L^T(y-\yb)-\xb}\leq \epsilon\sqrt{n(1+\norm{C_L}^2)}\norm{(x,y)-(\xb,\yb)}\]
  follows.
\end{proof}
Given $\eta,\kappa>0$, we now define for $x\in\R^n$ the set
\[\G^{\eta,\kappa}_{F,\xb}(x):=\{(\hat x,\hat y,L)\mv (\hat x,\hat y)\in\gph F,\ \norm{(\hat x,\hat y)-(\xb,0)}\leq \eta \norm{x-\xb}, L\in \Sp^*F(\hat x,\hat y)\cap\Z_n^{\rm reg}, \norm{C_L}\leq \kappa\}.\]

  \begin{theorem}\label{ThConvNewton}
  Assume that $F$ is \SCD \ssstar at $(\xb,0)\in\gph F$ and assume that there are  $\eta,\kappa>0$ such
  that for every $x\not\in F^{-1}(0)$ sufficiently close to $\xb$ we have
  $\G_{F,\xb}^{L,\kappa}(x)\not=\emptyset$. Then there exists some $\delta>0$ such that for every
  starting point $x^{(0)}\in\B_\delta(\xb)$ Algorithm \ref{AlgNewton} either stops after
  finitely many iterations at a solution or produces a sequence $x^{(k)}$ which converges
  superlinearly to $\xb$, provided we choose in every iteration $(\hat x^{(k)},\hat y^{(k)},L^{(k)})\in
  \G_{F,\xb}^{\eta,\kappa}(x^{(k)})$.
\end{theorem}
\begin{proof}
  Using Proposition \ref{PropConvNewton} with $\yb=0$, we can find some $\bar\delta>0$ such that \eqref{EqBndNewtonStep}
  holds with $\epsilon=\frac 1{2\eta\sqrt{n(1+\kappa^2)}}$ for all $(x,y)\in \gph F\cap \B_{\bar\delta}(\xb,0)$ and all
  $L\in\Sp^*F(x,y)\cap\Z_n^{\rm reg}$. Set $\delta:=\bar\delta/\eta$ and consider an iterate
  $x^{(k)}\in\B_\delta(\xb)\not\in F^{-1}(0)$. Then
  \[\norm{(\hat x^{(k)},\hat y^{(k)})-(\xb,0)}\leq \eta\norm{x^{(k)}-\xb}\leq\bar\delta\]
  and consequently
  \[\norm{x^{(k+1)}-\xb}\leq \frac 1{2\eta\sqrt{n(1+\kappa^2)}}\sqrt{n(1+\kappa^2)}\norm{(\hat x^{(k)},\hat y^{(k)})-(\xb,0)}
  \leq  \frac 12 \norm{x^{(k)}-\xb}\]
  by Proposition \ref{PropConvNewton}.
  It follows that for every starting point $x^{(0)}\in\B_\delta(\xb)$  Algorithm \ref{AlgNewton}
  either stops after finitely many iterations with a solution or produces a sequence $x^{(k)}$
  converging to $\xb$. The superlinear convergence of the sequence $x^{(k)}$ is now an easy
  consequence of Proposition \ref{PropConvNewton}.
\end{proof}
So far Algorithm \ref{AlgNewton} is only a straightforward adaption of the \ssstar Newton method from \cite{GfrOut21} to \SCD mappings. However, in \cite{GfrOut21} the \ssstar Newton method was only guaranteed to converge under the assumption of strong metric regularity, whereas we will now prove that for its \SCD variant a less restrictive condition is sufficient.

\begin{proposition}Let $F:\R^n\tto\R^n$ be \SCD regular around $(\xb,0)\in\gph F$. Then for every $\eta>0$ and every $\kappa>{\rm scd\,reg\;}F(\xb,0)$ there is a neighborhood $U$ of $\xb$
such that for every $x\in U$ the set $\G^{\eta,\kappa}_{F,\xb}(x)$ is nonempty and amounts to
\begin{equation}\label{EqSSNewtonG}\G^{\eta,\kappa}_{F,\xb}(x)=\big\{(\hat x,\hat y,L)\mv (\hat x,\hat y)\in\gph F,\ \norm{(\hat x,\hat y)-(\xb,0)}\leq \eta \norm{x-\xb}, L\in \Sp^*F(\hat x,\hat y)\big\}.\end{equation}
\end{proposition}
\begin{proof}
  By Proposition \ref{PropSCDReg} we can find some positive radius $\rho$ such that $F$ is \SCD regular around $(x,y)$ with modulus ${\rm scd\,reg\;}F(x,y)\leq\kappa$ for every $(x,y)\in\gph F\cap \B_\rho(\xb,0)$. By taking $U:=\B_{\rho/\eta}(\xb)$, for every $x\in U$ and every $(\hat x,\hat y)\in\gph F$ with $\norm{(\hat x,\hat y)-(\xb,0)}\leq \eta \norm{x-\xb}$ we have $(\hat x,\hat y)\in\B_{\rho}(\xb,0)$. By \eqref{EqSCDRegMod} we obtain that $\norm{C_L}\leq\kappa$ whenever $L\in \Sp^*F(\hat x,\hat y)$ and the assertion follows.
\end{proof}
Since  in \eqref{EqSSNewtonG} the right hand side does not depend on $\kappa$, we obtain the following corollary of Theorem \ref{ThConvNewton}.
\begin{corollary}\label{CorConvNewton}
 Assume that $F$ is \SCD \ssstar at $(\xb,0)\in\gph F$ and  \SCD regular around $(\xb,0)$. Then for every $\eta>0$ there is a neighborhood $U$ of $\xb$ such that
 for every starting point $x^{(0)}\in U$ Algorithm \ref{AlgNewton} is well-defined and either stops after finitely many iterations at a solution of \eqref{EqIncl} or produces a sequence $x^{(k)}$ converging superlinearly to $\xb$  for any choice of $(\hat x^{(k)} ,\hat y^{(k)})$ satisfying \eqref{EqBndApprStep} and any $L^{(k)}\in\Sp^*F(\hat x^{(k)} ,\hat y^{(k)})$.
\end{corollary}
\begin{remark}Note that Corollary \ref{CorConvNewton} guarantees not only locally superlinear convergence, but also that the method is locally well-defined, which is an advantage in comparison with the Josephy-Newton method from \cite{Jo79a}.
  In Theorem \ref{ThSS_StrMetrSubreg} below we will show that, under the assumptions of Corollary \ref{CorConvNewton}, the mapping $F$ is strongly metrically subregular at $(\xb,0)$. By \cite[Theorem 6.1]{CiDoKru18}, in such a case the convergence of the  Josephy-Newton method is also locally superlinear, provided the method is well-defined. This, however, need not be the case as illustrated in \cite[Example 5.13]{GfrOut21}, where the assumptions of Corollary \ref{CorConvNewton} are fulfilled, the \ssstar Newton method works well, but the Jospehy-Newton method collapses.
\end{remark}

\section{\label{SecStrMetrSubr}Strong metric subregularity on a neighborhood}
We first present a characterization of strong metric subregularity on a neighborhood, cf. Definition \ref{DefLocStrSubReg},  by means of the outer limiting graphical derivative defined in Definition \ref{DefLimGrDer}.
\begin{theorem}\label{ThLocStrSubreg}Consider a mapping  $F:\R^n\tto\R^n$  and  let $(\xb,\yb)\in\gph F$. Then $F$ is strongly metrically subregular around $(\xb,\yb)$ if and only if the condition
\begin{equation}\label{EqLocStrSubregCrit}
  0\in D^\sharp F(\xb,\yb)(u)\ \Rightarrow u=0
\end{equation}
holds and in this case one has
\begin{equation}\label{EqLocStrSubregMod}
  \lsubreg F(\xb,\yb)=\sup\{\norm{u}\mv  (u,v)\in \gph D^\sharp F(\xb,\yb),\ \norm{v}\leq 1\}.
\end{equation}
\end{theorem}
\begin{proof}
  We prove the ''if''-part by contraposition. Assume that \eqref{EqLocStrSubregCrit} holds but $F$ is not strongly metrically subregular around $(\xb,\yb)$. Then we  can find a sequence $(x_k,y_k)\longsetto{\gph F}(\xb,\yb)$ such that either $F$ is not strongly metrically subregular at $(x_k,y_k)$ for infinitely many $k$ or $\limsup_{k\to\infty}\subreg F(x_k,y_k)=\infty$. After possibly passing to some subsequence and taking into account Theorem \ref{ThCharRegByDer}, in both cases there is a sequence $\kappa_k\to\infty$ and  $(u_k,v_k)\in\gph DF(x_k,y_k)$ with $\norm{v_k}\leq 1$ such that $\norm{u_k}>\kappa_k \norm{v_k}$. Defining $(\tilde u_k,\tilde v_k):=(u_k,v_k)/\norm{u_k}\in\gph DF(x_k,y_k)$, we have $\norm{\tilde v_k}\leq 1/\kappa_k$ implying $\lim_{k\to\infty}\tilde v_k=0$. By possibly passing to a subsequence once more, $\tilde u_k$ converges to some $u$ with $\norm{u}=1$ and from the definition of $D^\sharp F(\xb,\yb)$ we obtain $(u,0)\in\gph D^\sharp F(\xb,\yb)$ contradicting \eqref{EqLocStrSubregCrit}. This proves the ''if''-part.

   In order to show the ''only if''-part assume that \eqref{EqLocStrSubregCrit} does not hold, so that there is some $u\not=0$ with $(u,0)\in \gph D^\sharp F(\xb,\yb)$. By definition of $D^\sharp F(\xb,\yb)$ there are sequences $(x_k,y_k)\longsetto{\gph F}(\xb,\yb)$ and $(u_k,v_k)\to(u,0)$  with $(u_k,v_k)\in\gph DF(x_k,y_k)$. If $F$ is not strongly metrically subregular at $(x_k,y_k)$ for infinitely many $k$, then it is not strongly metrically subregular around $(\xb,\yb)$ by definition. On the other hand, if $F$ is strongly metrically subregular at $(x_k,y_k)$ then $v_k\not=0$ and $\subreg F(x_k,y_k)\geq \norm{u_k}/\norm{v_k}$, which follows from Theorem \ref{ThCharRegByDer}. Hence, $\limsup_{k\to\infty}\subreg F(x_k,y_k)=\infty$ and $F$ is again not strongly metrically subregular around $(\xb,\yb)$. This proves the ''only  if''-part. There remains to show \eqref{EqLocStrSubregMod}. By definition we have
  \[\lsubreg F(\xb,\yb)=\limsup_{(x,y)\longsetto{\gph F}(\xb,\yb)}\sup\{\norm{u}\mv (u,v)\in\gph DF(x,y), \norm{v}\leq 1\}\]
  and therefore there are sequences $(x_k,y_k)\longsetto{\gph F}(\xb,\yb)$ and $(u_k,v_k)\in \gph DF(x_k,y_k)$ with $\norm{v_k}\leq 1$ and $\norm{u_k}\to \lsubreg F(\xb,\yb)<\infty$. By possibly passing to a subsequence, $(u_k,v_k)$ converges to some $(u,v)$ with $\norm{v}\leq 1$. By definition of $D^\sharp F$ we have $(u,v)\in \gph D^\sharp F(\xb,\yb)$ and
  $\lsubreg F(\xb,\yb)=\norm{u}\leq \xi:= \sup\{\norm{u}\mv \exists v:\ (u,v)\in \gph D^\sharp F(\xb,\yb),\ \norm{v}\leq 1\}$ follows. Next consider a sequence $(u_k,v_k)\in \gph D^\sharp F(\xb,\yb)$ with $\norm{v_k}\leq 1$ and $\norm{u_k}\to \xi$. Then for every $k$ there are $(x_k,y_k)\in\gph F$ and $(u_k',v_k')\in \gph DF(x_k,y_k)$ such that $\norm{(x_k,y_k)-(\xb,\yb)}\leq \frac 1k$ and $(\norm{u_k',v_k')-(u_k,v_k)}\leq \frac 1k$ and
  \[\subreg F(x_k,y_k)\geq \xi_k:=\begin{cases}\frac {\norm{u_k'}}{\norm{v_k'}}&\mbox{if $\norm{v_k'}>1$}\\
  \norm{u_k'}&\mbox{if $\norm{v_k'}\leq1$}\end{cases}\]
  follows. In case when $\norm{v_k'}>1$ we have
  \[\vert\frac {\norm{u_k'}}{\norm{v_k'}}-\norm{u_k'}\vert =\frac{\norm{v_k'}-1}{\norm{v_k'}}\norm{u_k'}\leq \frac 1{k+1} \norm{u_k'}\]
  and $\lim_{k\to\infty}\xi_k=\xi$ follows. Hence $\lsubreg F(\xb,\yb)\geq\limsup_{k\to\infty}\xi_k=\xi$ and relation \eqref{EqLocStrSubregMod} is established.
\end{proof}
From Lemma \ref{LemSCDname1} together with Lemma \ref{LemSCCRegPrimal}, we may  conclude  that strong metric subregularity around $(\xb,\yb)\in\gph F$ implies \SCD regularity around $(\xb,\yb)$
and that $\lsubreg F(\xb,\yb)\geq {\rm scd\,reg\;}F(\xb,\yb)$. Next we  show  that, conversely,
\SCD regularity in conjunction with \SCD semismoothness*  provides a sufficient condition for strong metric subregularity.
\begin{theorem}\label{ThSS_StrMetrSubreg}
  Assume that $F:\R^n\tto\R^n$ is \SCD regular around $(\xb,\yb)\in\gph F$. Then for every $\kappa>{\rm scd\,reg\;}F(\xb,\yb)$ there is a neighborhood $W$ of $(\xb,\yb)$ such that $F$ is strongly metrically subregular  with modulus $\subreg F(x,y)<\kappa$ at every point $(x,y)\in\gph F\cap W$ where $F$ is \SCD \ssstar.
\end{theorem}
\begin{proof}
  Fixing $\kappa>\tilde \kappa>{\rm scd\,reg\;}F(\xb,\yb)$, by Proposition \ref{PropSCDReg} there is an open neighborhood $W$ of $(\xb,\yb)$ such that for every $(x,y)\in\gph F\cap W$ the mapping $F$ is \SCD regular around $(x,y)$ with modulus ${\rm scd\,reg\;}F(x,y)\leq\tilde \kappa$ . Consider $(\tilde x,\tilde y)\in\gph F\cap W$ where $F$ is \SCD \ssstar. Assume now that $F$ is not metrically subregular  at $(\tilde x,\tilde y)$ or that $\subreg F(\tilde x,\tilde y)>\tilde\kappa$. Then there is some $\kappa'>\tilde\kappa$ and a sequence $x_k$ converging to $\tilde x$ such that $\dist{x_k,F^{-1}(\tilde y)}>\kappa'\dist{\tilde y, F(x_k)}$ $\forall k$. Consider $y_k\in F(x_k)$ with $\dist{\tilde y, F(x_k)}=\norm{y_k-\tilde y}$. Then $y_k$ converges to $\tilde y$ and for all $k$ sufficiently large we have $(x_k,y_k)\in \gph F\cap W$. Pick $L_k\in \Sp^*F(x_k,y_k)$. Using Proposition \ref{PropConvNewton} and \eqref{EqSCDRegMod} we have that $\norm{C_{L_k}}\leq\tilde\kappa$  and
  \[\norm{x_k-C_{L_k}^T(y_k-\tilde y)-\tilde x}\leq \frac 1k \sqrt{n(1+\norm{C_{L_k}}^2)}\norm{(x_k-\tilde x,y_k-\tilde y)}\leq \frac 1k \sqrt{n(1+\tilde\kappa^2)}(\norm{x_k-\tilde x}+\norm{y_k-\tilde y})\]
  implying
  \[(1-\alpha_k)\norm{x_k-\tilde x}\leq \norm{C_{L_k}(y_k-\tilde y)}+ \alpha_k\norm{y_k-\tilde y}\leq (\tilde\kappa+\alpha_k)\norm{y_k-\tilde y},\]
     where $\alpha_k:= \frac 1k \sqrt{n(1+\tilde\kappa^2)}$. Since $\alpha_k\to 0$ as $k\to\infty$, we have $(\tilde\kappa+\alpha_k)/(1-\alpha_k)<\kappa'$ for all $k$ sufficiently large and therefore $\norm{x_k-\tilde x}<\kappa'\norm{y_k-\tilde y}$ in contrary to our assumption. This  shows that $F$ is metrically subregular at $(\tilde x,\tilde y)$ and $\subreg F(\tilde x,\tilde y)\leq \tilde\kappa<\kappa$. Further $\tilde x$ must be an isolated point in $F^{-1}(\tilde y)$. Assume on the contrary that there is a sequence $x_k\in F^{-1}(\tilde y)$ converging to $\tilde x$. Taking $L_k\in \Sp^*F(x_k,\tilde y)$ and applying Proposition \ref{PropConvNewton} with $\epsilon =1/(2\sqrt{n(1+\tilde\kappa^2)}$, we obtain for all $k$ sufficiently large
  \[\norm{x_k-\tilde x}=\norm{x_k-C_{L_k}^T(\tilde y-\tilde y)-\tilde x}\leq \frac 12\norm{(x_k-\tilde x,\tilde y-\tilde y)}=\frac 12\norm{x_k-\tilde x},\]
  a contradiction. This shows that $F$ is even strongly metrically subregular at $(\tilde x,\tilde y)$.
\end{proof}
\begin{remark}
  In the special case of a single-valued locally Lipschitzian mapping $F:\R^n\to\R^n$ the statement of Theorem \ref{ThSS_StrMetrSubreg} can be derived also from \cite[Proposition 1]{Gow04}.
\end{remark}
For \ssstar mappings we arrive thus at the following equivalence.
\begin{corollary}\label{CorStrMetrSubreg}
  Assume that $F:\R^n\tto\R^n$  is \SCD \ssstar around $(\xb,\yb)\in\gph F$. Then $F$ is strongly metrically subregular around $(\xb,\yb)$ if and only if $F$ is \SCD regular around $(\xb,\yb)$ and in this case one has $\lsubreg F(\xb,\yb)= \scdreg F(\xb,\yb)$.
\end{corollary}
\begin{example}\label{ExSubspReg}
Consider the mapping $F:=\R^2\tto\R^2$ given by
\[F(x_1,x_2):=\myvec{x_1\\-x_2}+h(x)+N_C(x_1,x_2),\]
where $C:=\{(x_1,x_2)\mv -\frac 12 x_1\leq x_2\leq \frac 12 x_1\}$ is a convex polyhedral cone and $h:\R^2\to\R^2$ is any continuously differentiable mapping satisfying $h(0,0)=(0,0)$, $\nabla h(0,0)=0$. As the reference point we take $(\xb,\yb)=\big((0,0),(0,0)\big)$.
The mapping $N_C$ is a polyhedral mapping and therefore \ssstar at any point of its graph by \cite{GfrOut21}. Further, $N_C(x)=\partial\delta_C(x)$ is an SCD mapping by Corollary \ref{CorSCDConv}. Thus, $F$ is both an \SCD mapping around and \ssstar at any point of its graph, because it differs from $N_C$  by a continuously differentiable mapping. Now let us calculate $\Sp^*F(\xb,\yb)$. The critical cone $\K_C(\xb,\yb)$ amounts to $C$ and has therefore the 4 faces $C$, $\{(u,\frac 12 u)\mv u\geq 0\}$, $\{(u,-\frac 12 u)\mv u\geq 0\}$ and $\{(0,0)\}$. By using Example \ref{ExPoly} we conclude that $\Sp^*N_C(0,0)$ consists of the 4 subspaces $L_1:=\R^2\times\{(0,0)\}$, $L_2:=\{((u,\frac 12 u),(-\frac 12 v,v))\mv (u,v)\in\R^2\}$, $L_3:=\{((u,-\frac 12 u),(\frac 12 v,v))\mv (u,v)\in\R^2\}$ and $L_4:=\{(0,0)\}\times\R^2$ and Proposition \ref{PropSum} tells us that $\Sp^*F(\xb,\yb)=\{TL_1, TL_2, TL_3,TL_4\}$ where
\[T=\left(\begin{matrix}1&0&0&0\\ 0&1&0&0\\1&0&1&0\\0&-1&0&1
\end{matrix} \right).\]
Straightforward calculations yield
\begin{align*}&TL_1=\{((u,v),(u,-v))\mv (u,v)\in\R^2\}, TL_2=((u,\frac 12 u),(u-\frac 12v,-\frac 12 u+v))\mv (u,v)\in\R^2\},\\
 &TL_3=\{((u,-\frac12 u),(u+\frac 12 v, \frac 12 u+v))\mv (u,v)\in\R^2\},\ TL_4=\{((0,0),(u,v))\mv (u,v)\in\R^2\}.
 \end{align*}
 Now it easily follows that $F$ is \SCD regular around $(\xb,\yb)$ with
 \begin{align*}
   C_{TL_1}=\left(\begin{matrix}1&0\\0&-1\end{matrix}\right),\
   C_{TL_2}=\left(\begin{matrix}\frac 43&\frac 23\\\frac 23&\frac 13\end{matrix}\right),\
   C_{TL_3}=\left(\begin{matrix}\frac 43&-\frac 23\\-\frac 23&\frac 13\end{matrix}\right),\
   C_{TL_4}=\left(\begin{matrix}0&0\\0&0\end{matrix}\right)
 \end{align*}
 and $\norm{C_{TL_1}}=1$, $\norm{C_{TL_2}}=\norm{C_{TL_3}}=\frac 53$, $\norm{C_{TL_4}}=0$. Hence, by virtue of Theorem \ref{ThSS_StrMetrSubreg}, $F$ is strongly metrically subregular around $(\xb,\yb)$ with modulus $\lsubreg F(\xb,\yb)=\frac 53$.

To illustrate this result, we explicitly compute $F^{-1}$ in case $h=0$. One obtains that
\begin{equation}\label{EqExFInv}
  F^{-1}(y)=\begin{cases}
\big\{z_1(y)\big\}&\mbox{if $-\frac 12 y_1+y_2>  0,\;2y_1+y_2 \geq 0$,}\\
\big\{z_1(y),z_2(y),z_3(y)\big\}&\mbox{if $-\frac 12 y_1+y_2\leq 0,\;-\frac 12 y_1-y_2\leq 0$,}\\
\big\{z_3(y)\big\}&\mbox{if $-\frac 12 y_1-y_2>0,\;2y_1-y_2\geq  0$,}\\
\big\{z_4(y)\big\}&\mbox{if $2y_1+y_2\le 0,\;2y_1-y_2\le 0\big\}$,}\\
\end{cases}
\end{equation}
with the mappings $z_i(y)$, $i=1,\ldots,4$, in \eqref{EqExFInv} specified via
\begin{eqnarray*}
z_1(y)&:=&\Big(\frac 43 y_1+\frac 23 y_2,\frac 23 y_1+\frac 13 y_2\Big),\;z_2(y):=(y_1,-y_2),\\
z_3(y)&:=&\Big(\frac 43 y_1-\frac 23 y_2,-\frac 23 y_1+\frac 13 y_2\Big),\;z_4(y):=(0,0).
\end{eqnarray*}
We see that $F^{-1}$ has the isolated calmness property at every point of its graph close to $(0,0)$, but it is not single-valued.
\end{example}
\section{On strong metric regularity}
Our results on strong metric regularity pertain again mappings $F:\R^n\tto\R^n$ and  are partly expressed in terms of certain bases for the subspaces $L\in \Sp F(\xb,\yb)$.

Given a mapping $F:\R^n\tto\R^n$ which is graphically Lipschitzian of dimension $n$ at $(\xb,\yb)\in\gph F$ with transformation mapping $\Phi$ according to Definition \ref{DefGraphLip}, we denote by
$\overline{\nabla}{}^\Phi F(\xb,\yb)$ the collection of all $2n\times n$ matrices $Z$ such that $\rge Z\in \Sp F(\xb,\yb)$ and $\pi_1(\nabla \Phi(\xb,\yb)Z)=I$.

Note that by Proposition \ref{PropGraphLipSCD} for every $L\in \Sp F(\xb,\yb)$ there exists a unique $Z\in \overline{\nabla}{}^\Phi F(\xb,\yb)\cap \M(L)$. Since for every Lipschitzian mapping $f:U\to\R^n$, $U\subset\R^n$ open, the B-subdifferential $\overline{\nabla} f(u)$ is compact for every $u\in U$, we conclude from \eqref{EqGraphLipSCD1} that $\overline{\nabla}{}^\Phi F(\xb,\yb)$ is compact as well.

The next statement provides us with an upper approximation of the graphs of the strict derivative and the limiting coderivative, respectively.
\begin{proposition}\label{PropInclStrDer}
  Let $F:\R^n\tto\R^n$ be graphically Lipschitzian at $(\xb,\yb)\in\gph F$ with transformation mapping $\Phi$.
  Then
  \begin{align}
    \label{EqInclStrictDer}&\gph D_*F(\xb,\yb)\subseteq \bigcup_{Z\in\co\overline{\nabla} ^\Phi F(\xb,\yb)}\rge Z,\\
    \label{EqInclCoDer}&\gph D^*F(\xb,\yb)\subseteq \bigcup_{Z\in\co\overline{\nabla} ^\Phi F(\xb,\yb)}(\rge Z)^*
  \end{align}
\end{proposition}
\begin{proof}
  According to Definition \ref{DefGraphLip} consider the open neighborhoods $W$ of $(\xb,\yb)$, $U$ of $\bar w$ and the Lipschtitzian mapping $f:U\to\R^n$ with $\Phi(\gph F\cap W)=\gph f$, where $\bar w=\pi_1(\Phi(\xb,\yb))$. Consider $(u,v)\in \gph D_*F(\xb,\yb)$ together with sequences $(x^1_k,y^1_k)\longsetto{\gph F} (\xb,\yb)$, $(x^2_k,y^2_k)\longsetto{\gph F} (\xb,\yb)$ and $t_k\downarrow 0$ such that $(u,v)=\lim_{k\to\infty}(x^2_k-x^1_k,y^2_k-y^1_k)/t_k$. For each $k$ let $w^i_k$, $i=1,2$ be given by $w^i_k=\pi_1(\Phi(x^i_k,y^i_k))$.
  Then
  \begin{align*}(w^2_k-w^1_k, f(w^2_k)-f(w^1_k))&=\Phi(x^2_k,y^2_k)-\Phi(x^1_k,y^1_k)=\nabla \Phi(\xb,\yb)(x^2_k-x^1_k,y^2_k-y^1_k)+ \oo(\norm{(x^2_k-x^1_k,y^2_k-y^1_k)})\\
  &=\nabla \Phi(\xb,\yb)(x^2_k-x^1_k,y^2_k-y^1_k)+ \oo(t_k)\end{align*}
  implying that
  \[\lim_{k\to\infty}\frac{(w^2_k-w^1_k, f(w^2_k)-f(w^1_k))}{t_k}=\lim_{k\to\infty}\nabla\Phi(\xb,\yb)\frac{(x^2_k-x^1_k,y^2_k-y^1_k)}{t_k}=\nabla\Phi(\xb)(u,v)\in \gph D_*f(\bar w).\]
  Hence $\pi_2(\nabla\Phi(\xb)(u,v))\in D_*f(\bar w)(\pi_1(\nabla\Phi(\xb)(u,v)))$ and by \cite[Theorem 9.62]{RoWe98} there is some $B\in \co\overline{\nabla} f(\bar w)$ satisfying $\pi_2(\nabla\Phi(\xb)(u,v))= B\pi_1(\nabla\Phi(\xb)(u,v))$ which is the same as $\nabla\Phi(\xb)(u,v)\in \rge (I,B)$. $B$ can be expressed  as  a convex combination $\sum_{i=1}^N\alpha_i B_i$ with $B_i\in \overline{\nabla} f(\bar w)$, $\alpha_i\geq 0$, $\sum_{i=1}^N\alpha_i=1$ and therefore
  \[(u,v)\in \nabla \Phi(\xb,\yb)^{-1}\rge(I,B)=\Phi(\xb,\yb)^{-1}\rge\Big[\sum_{i=1}^N\alpha_i\myvec{I\\B_i}\Big]=\rge\Big[\sum_{i=1}^N\alpha_i \Phi(\xb,\yb)^{-1}\myvec{I\\B_i}\Big].\]
  Denoting $Z_i:=\Phi(\xb,\yb)^{-1}\myvec{I\\B_i}$ we have $\rge Z_i\in\Sp F(\xb,\yb)$ by \eqref{EqGraphLipSCD1} and $\pi_1(\Phi(\xb,\yb)Z_i)=I$ yielding $Z_i\in\overline{\nabla} ^\Phi F(\xb,\yb)$. Thus $(u,v)\in\rge Z$ with $Z=\sum_{i=1}^N\alpha_i Z_i\in\co\overline{\nabla}{}^\Phi F(\xb,\yb)$ verifying \eqref{EqInclStrictDer}.

  Now consider $(y^*,x^*)\in \gph D^*F(\xb,\yb)$ which is the same as
  \[S_n^T(y^*,x^*)\in N_{\gph F}(\xb,\yb)=\nabla \Phi(\xb,\yb)^TN_{\gph f}(\bar w,f(\bar w),\]
   where the second equality follows from \cite[Exercise 6.7]{RoWe98}. Hence $z^*:=S_n\nabla \Phi(\xb,\yb)^{-T}S_n^T(y^*,x^*)\in S_n N_{\gph f}(\bar w,f(\bar w)= \gph D^*f(\bar w)$ implying $\pi_2(z^*)\in D^*f(\bar w)(\pi_1(z^*))$. By \cite[Theorem 9.62]{RoWe98} there is some $B\in \co\overline{\nabla} f(\bar w)$ such that $\pi_2(z^*)=B^T\pi_1(z^*)$ which is the same as $z^*\in\rge(I,B^T)$ and
  \[(y^*,x^*)\in\rge\Big[S_n\nabla \Phi(\xb,\yb)^TS_n^T\myvec{I\\B^T}\Big]\]
  follows. Taking into account that
  \[\rge\Big[S_n\nabla \Phi(\xb,\yb)^TS_n^T\myvec{I\\B^T}\Big]^\perp=\rge\Big[S_n\nabla \Phi(\xb,\yb)^{-1}S_n^T\myvec{B\\-I}\Big]=\rge\Big[-S_n\nabla \Phi(\xb,\yb)^{-1}\myvec{I\\B}\Big],\]
  we obtain
  \[\rge\Big[S_n\nabla \Phi(\xb,\yb)^TS_n^T\myvec{I\\B^T}\Big]^*=S_n\rge\Big[-S_n\nabla \Phi(\xb,\yb)^{-1}\myvec{I\\B}\Big]=\rge\Big[\nabla \Phi(\xb,\yb)^{-1}\myvec{I\\B}\Big].\]
  As we have shown above, the latter subspace equals to $\rge Z$ with $Z\in \co\overline{\nabla} ^\Phi F(\xb,\yb)$ and   $(y^*,x^*)\in (\rge Z)^*$ follows.
\end{proof}
On the basis of Proposition \ref{PropInclStrDer} we can now establish the following characterization of strong metric regularity.
\begin{theorem}\label{ThGenClarkeInv}
  Consider a mapping  $F:\R^n\tto\R^n$  and let $(\xb,\yb)\in\gph F$.
  \begin{enumerate}
    \item[(i)]If $F$ is strongly metrically regular around $(\xb,\yb)$ then it is graphically Lipschitzian of dimension $n$ at $(\xb,\yb)$ with transformation mapping $\Phi(x,y)=(y,x)$ and one has that $\{\rge Z\mv Z\in\co\overline{\nabla}{}^\Phi F(\xb,\yb)\}\subseteq \Z_n^{\rm reg}$. Further, $F$ is \SCD regular around $(\xb,\yb)$ and  $\reg F(\xb,\yb)=\scdreg F(\xb,\yb)$.
    \item[(ii)]Conversely, if $F$ is graphically Lipschitzian of dimension $n$ at $(\xb,\yb)$ with some transformation mapping $\Phi$ such that $\{\rge Z\mv Z\in\co\overline{\nabla}{}^\Phi F(\xb,\yb)\}\subseteq \Z_n^{\rm reg}$ then $F$ is strongly metrically regular around $(\xb,\yb)$.
  \end{enumerate}
\end{theorem}
\begin{proof}If $F$ is strongly metrically regular around $(\xb,\yb)$, then by Theorem \ref{ThStrMetrReg} it is clearly graphically Lipschitzian with the given transformation mapping $\Phi$  and from Proposition \ref{PropGraphLipSCD} we obtain
$\Sp F(\xb,\yb)=\{\rge(B,I)\mv B\in\overline{\nabla} f(\yb)\}$ where $f$ denotes the Lipschitz continuous localization of $F^{-1}$ around $(\yb,\xb)$. Thus
\[\overline{\nabla}{}^\Phi F(\xb,\yb)=\Big\{\myvec{B\\I}\mv  B\in\overline{\nabla} f(\yb)\Big\}\]
and consequently every matrix $Z\in \co \overline{\nabla}{}^\Phi F(\xb,\yb)$ is of the form $Z=\myvec{B\\I}$ with $B\in \co \overline{\nabla} f(\yb)$. From this we can easily deduce that $L:=\rge Z\in\Z_n^{\rm reg}$ and $B=C_L$ showing that $F$ is \SCD regular. In order to verify the formula for the modulus of strong metric regularity we use \eqref{EqModStrMetrReg}. Let $\epsilon>0$ and consider $(u,v)\in\gph D_*F(\xb,\yb)$ with $\norm{v}\leq 1$ and $\norm{u}\geq \reg F(\xb,\yb)-\epsilon$. By Proposition \ref{PropInclStrDer} there is some $Z\in \co \overline{\nabla}{}^\Phi F(\xb,\yb)$ and some $w\in \R^n$ with $(u,v)=Zw$. Thus there is $\bar B\in \co \overline{\nabla} f(\yb)$ such that $u=\bar Bw$ and $v=w$ yielding
\begin{align*}\reg F(\xb,\yb)-\epsilon&\leq\norm{u}\leq \norm{\bar B}\norm{v}\leq\norm{\bar B}\leq\sup\{\norm{B}\mv B\in \co\overline{\nabla} f(\yb)\}\\
&=\sup\{\norm{B}\mv B\in \overline{\nabla} f(\yb)\}=\sup\{\norm{C_L}\mv L\in\Sp F(\xb,\yb)\}=\scdreg F(\xb,\yb)
\end{align*}
by Lemma \ref{LemSCCRegPrimal}.
Since $\epsilon>0$ can be chosen arbitrarily small, there holds $\reg F(\xb,\yb)\leq \scdreg F(\xb,\yb)$. On the other hand,  we have $\bigcup\Sp F(\xb,\yb)\subseteq \gph D^\sharp F(\xb,\yb)\subseteq \gph D_*F(\xb,\yb)$ implying $\scdreg F(\xb,\yb)\leq \reg F(\xb,\yb)$ by Lemma \ref{LemSCCRegPrimal} and \eqref{EqModStrMetrReg}. This proves (i).

The statement (ii) follows from Theorem \ref{ThCharRegByDer} together with Proposition \ref{PropInclStrDer}. If $0\in \gph D_*F(\xb,\yb)(u)$ then there is some $Z\in \co \overline{\nabla}{}^\Phi F(\xb,\yb)$ such that $(u,0)\in\rge Z$ and $u=0$ follows from $\rge Z\in\Z_n^{\rm reg}$. Similarly, if $0\in D^*F(\xb,\yb)(y^*)$ then there is some $Z\in \co \overline{\nabla}{}^\Phi F(\xb,\yb)$ such that $(y^*,0)\in(\rge Z)^*$. Since $\rge Z\in\Z_n^{\rm reg}$, we have $(\rge Z)^*\in\Z_n^{\rm reg}$ by Proposition \ref{PropC_L} and $y^*=0$ follows. Hence, both \eqref{EqStrictCrit} and \eqref{EqMoCrit} are fulfilled and strong metric regularity of $F$ has been established.
\end{proof}

  Consider the special case of a single-valued Lipschitzian mapping $F:\R^n\to\R^n$ so that $\Phi(x,y)=(x,y)$. Then $\overline{\nabla}{}^\Phi F(x,F(x))=\{(I,B)\mv B\in\overline{\nabla} F(x)\}$ by Lemma \ref{LemSCDSingleValued} and therefore
  \[\co \overline{\nabla}{}^\Phi F(x,F(x))=\{(I,B)\mv B\in\co\overline{\nabla} F(x)\}.\]
  Thus the requirement in Theorem \ref{ThGenClarkeInv}(ii) that $\{\rge Z\mv Z\in\co\overline{\nabla}{}^\Phi F(\xb,\yb)\}\subseteq \Z_n^{\rm reg}$ is equivalent to the condition that every matrix $B$ belonging to Clarke's generalized Jacobian $\co\overline{\nabla} F(x)$ is nonsingular. Therefore we may consider Theorem \ref{ThGenClarkeInv}(ii) as a generalization of Clarke's Inverse Function Theorem, see, e.g., \cite[Theorem 7.1.1]{Cla83}, to set-valued mappings.

Note that  $\reg F(\xb,\yb)=\scdreg F(\xb,\yb)$ whenever $F$ is strongly metrically regular around $(\xb,\yb)$.  This fact will not be repeated in the following results, where we present sufficient conditions for strong metric regularity.

The sufficient condition for strong metric regularity in Theorem \ref{ThGenClarkeInv}(ii) depends on the particular basis representation $\overline{\nabla}{}^\Phi F(\xb,\yb)$ of $\Sp F(\xb,\yb)$. The next results are stated in  terms of the elements $L\in\Sp F(\xb,\yb)$  which do not depend on a basis.
\begin{corollary}
  Given $F:\R^n\tto\R^n$ and  a point $(\xb,\yb)\in\gph F$, assume that $\Sp F(\xb,\yb)=\{\bar L\}$ is a singleton. Then $F$ is strongly metrically regular around $(\xb,\yb)$ if and only if $F$ is graphically Lipschitzian of dimension $n$ at $(\xb,\yb)$ and $\bar L\in\Z_n^{\rm reg}$.
\end{corollary}
When $\Sp F(\xb,\yb)$ is a singleton and $F$ is not graphically Lipschitzian at $(\xb,\yb)$ then $F$ cannot be strongly metrically regular by Theorem \ref{ThGenClarkeInv}(i). However, if $F$ is \SCD \ssstar at (around) $(\xb,\yb)$, then it is at least strongly metrically subregular at (around) $(\xb,\yb)$. Consider the following example.
\begin{example}
  Let $q:\R\to\R$ be given by $q(x)=\frac 23 {\rm sign\,}(x)\vert x\vert^{\frac 32}$. Then $\partial q(x)=\vert x\vert^{\frac 12}$ is not graphically Lipschitzian of dimension $1$ at $(0,0)$
  but it is an \SCD mapping and \SCD \ssstar. Further, $\Sp \partial q(0,0)=\{\{0\}\times \R\}$ is a singleton and clearly $\{0\}\times \R\in\Z_n^{\rm reg}$. Thus we deduce from Corollary \ref{CorStrMetrSubreg} that $\partial q$ is strongly metrically subregular around $(0,0)$. However, $\partial q$ is not strongly metrically regular around $(0,0)$ because $\partial q^{-1}(y)=\emptyset$ for every $y<0$.
\end{example}

We will now present a basis-independent characterization  of strong metric regularity for locally maximally hypomonotone mappings.
\begin{theorem}\label{ThStrRegHypo}
  Assume that $F:\R^n\tto\R^n$ is locally maximally hypomonotone at $(\xb,\yb)\in \gph F$. Then the following two statements are equivalent:
  \begin{enumerate}
  \item[(i)] $F$ is \SCD regular around $(\xb,\yb)$ and for every $L\in\Sp F(\xb,\yb)$ the matrix $C_L$ is positive semidefinite.
   \item[(ii)]$F$ is strongly metrically regular around $(\xb,\yb)$ and
  \begin{equation}\label{EqPseudoMon}
    \liminf_{(x^1,y^1),(x^2,y^2)\longsetto{\gph F}(\xb,\yb)}\frac{\skalp{x^1-x^2,y^1-y^2}}{\norm{x^1-x^2}\norm{y^1-y^2}}\geq 0
  \end{equation}
  with the convention $0/0:=0$.
  \end{enumerate}
\end{theorem}
\begin{proof}
We first prove (i)$\Rightarrow$(ii). By Corollary \ref{CorHypomonotone} there is some $\lambda\geq 1$ such that $F$ is graphically Lipschitzian at $(\xb,\yb)$ with transformation mapping $\Phi(x,y)=(\lambda x+ y,x)$ and for every $Z\in \overline{\nabla}{}^\Phi F(\xb,\yb)$ there is a firmly nonexpansive   $n\times n$ matrix $B$,  such that $Z=\myvec{B\\I- \lambda B}$. Since $F$ is \SCD regular around $(\xb,\yb)$, $\pi_2(Z)=I- \lambda B$ is nonsingular and $C_{\rge Z}=B(I- \lambda B)^{-1}$. Consider $u\in \R^n$ and set $v:=(I-\lambda B)u$. By the posed assumption, $B(I-\lambda B)^{-1}$ is positive semidefinite and we obtain
  \[0\leq \lambda\skalp{v,  B(I-\lambda B)^{-1}v}=\lambda \skalp{v, Bu}=\lambda\skalp{(I- \lambda B)u, Bu}\]
  implying $\skalp{u, \lambda Bu}\geq \norm{\lambda Bu}^2$. Thus $\lambda B$ is firmly nonexpansive and, consequently, $\norm{2\lambda B-I}\leq 1$. Since $I-\lambda B$ is nonsingular, we deduce from \cite[Theorem 3.3]{BauMofWa12} that $\norm{\lambda B}<1$.  Now consider $\bar Z\in \co \overline{\nabla}{}^\Phi F(\xb,\yb)$. It follows that $\bar Z=\myvec{\bar B\\I- \lambda \bar B}$, where $\bar B$ is some convex combination of  matrices $B_i$ with $\norm{\lambda B_i}<1$ and $\norm{2 \lambda B_i-I}\leq 1$. It follows that  $\norm{\lambda \bar B}<1$ and  $\norm{2 \lambda \bar B-I}\leq 1$. Thus $\pi_2(\bar Z)=I-\lambda \bar B$ is nonsingular and from Proposition \ref{PropC_L} we may deduce that $\rge\bar Z\in\Z_n^{\rm reg}$. Now strong regularity of $F$ follows from Theorem \ref{ThGenClarkeInv}(ii). Next we prove \eqref{EqPseudoMon} by contraposition. Assume on the contrary that there are sequences $(x^i_k,y^i_k)\longsetto{\gph F}(\xb,\yb)$, $i=1,2$, and some $\eta>0$ such that $\skalp{x^1_k-x^2_k,y_k^1-y_k^2}< -\eta \norm{x_k^1-x_k^2}\norm{y_k^1-y_k^2}$ for all $k$. By local hypomonotonicity,
  $\skalp{x^1_k-x^2_k,y^1_k-y^2_k}\geq -(\lambda-1)\norm{x^1_k-x^2_k}^2$ and therefore $\lambda>1$ and $\norm{x^1_k-x^2_k}\geq \frac \eta{\lambda-1}\norm{y_k^1-y_k^2}$. On the other hand, by strong metric regularity, choosing $c>\reg F(\xb,\yb)$, we have $\norm{x^1_k-x^2_k}\leq c\norm{y_k^1-y_k^2}$ for all $k$ sufficiently large. Let $t_k:=\norm{y_k^1-y_k^2}+ \norm{x^1_k-x^2_k}$. By possibly passing to a subsequence, $(x^1_k-x^2_k,y_k^1-y_k^2)/t_k$ converges to some $(u,v)\in D_*F(\xb,\yb)$ with $\norm{u}+\norm{v}=1$, $\frac \eta{\lambda-1}\norm{v}\leq \norm{u}\leq c\norm{v}$ and $\skalp{u,v}\leq -\eta\norm{u}\norm{v}$. We deduce that both $u$ and $v$ are nonzero and thus $\skalp{u,v}<0$. By Proposition \ref{PropInclStrDer} there is some $\bar Z\in \co \overline{\nabla}{}^\Phi F(\xb,\yb)$ and some $p\in\R^n$ with $(u,v)=\bar Z p$. As shown above, $\bar Z=\myvec{\bar B\\I- \lambda \bar B}$ for some $n\times n$ matrix $\bar B$ with $\norm{2\lambda \bar B -I}\leq 1$, i.e., $\lambda \bar B$ is firmly nonexpansive. It follows that
  \[0\leq \skalp{p,\lambda\bar Bp}-\norm{\lambda \bar Bp}^2=\lambda \skalp{\bar Bp, p-\lambda\bar Bp}=\lambda\skalp{u,v},\]
   contradicting $\skalp{u,v}<0$. Hence the implication (i)$\Rightarrow$(ii) is verified.

  To show the reverse implication note that strong metric regularity implies \SCD regularity by Theorem \ref{ThGenClarkeInv}. We prove that $C_L$ is positive semidefinite for every $L\in \Sp F(\xb,\yb)$ by contradiction. Assume that there exists some $L\in \Sp F(\xb,\yb)$ and $p\in\R^n$ with $\skalp{C_Lp,p}<0$. Since $(C_Lp,p)\in L\subseteq D^\sharp F(\xb,\yb)\subseteq D_*F(\xb,\yb)$, there are sequences $t_k\downarrow 0$, $(x_k,y_k)\longsetto{\gph F}(\xb,\yb)$ and $(u_k,v_k)\to(C_Lp,p)$ with $(x_k',y_k'):=(x_k,y_k)+t_k(u_k,v_k)$. It follows that $(x_k'-x_k)/\norm{x_k'-x_k}=u_k/\norm{u_k}\to C_Lp$, $(y_k'-y_k)/\norm{y_k'-y_k}=v_k/\norm{v_k}\to p/\norm{p}$ and therefore
  \begin{equation}\label{EqAuxPseudoMon}\lim_{k\to\infty}\frac{\skalp{x_k'-x_k,y_k'-y_k}}{\norm{x_k'-x_k}\norm{y_k'-y_k}}=\frac{\skalp{C_L p,p}}{\norm{C_Lp}\norm{p}}<0\end{equation}
  contradicting \eqref{EqPseudoMon}.
  \end{proof}

\begin{corollary}\label{CorStrRegMonMap}
  Let $F:\R^n\tto\R^n$ be locally monotone around $(\xb,\yb)\in\gph F$. Then the following statements are equivalent.
  \begin{enumerate}
    \item[(i)]$F$ is strongly metrically regular around $(\xb,\yb)$.
    \item[(ii)]$F$ is metrically regular around $(\xb,\yb)$.
    \item[(iii)]$F$ is \SCD regular around $(\xb,\yb)$ and locally maximally monotone at $(\xb,\yb)$.
  \end{enumerate}
  In this case, the matrices $C_L$, $L\in\Sp F(\xb,\yb)$, are positive semidefinite.
\end{corollary}
\begin{proof}
  The equivalence between (i) and (ii) follows from the definitions of (strong) metric regularity and \cite[Theorem 3G.5]{DoRo14}. In view of Theorem \ref{ThStrRegHypo}, in order to verify (i)$\Rightarrow$(iii), we only have to show that $F$ is locally maximally monotone. By taking into account that $\gph F^{-1}=\{(y,x)\mv (x,y)\in\gph F\}$, it follows readily from the definition that $F$ is locally maximally monotone at $(\xb,\yb)$ if and only if $F^{-1}$ is locally maximally monotone at $(\yb,\xb)$. $F^{-1}$ has a Lipschitz continuous monotone localization and is therefore locally maximally monotone  at $(\yb,\xb)$ by \cite[Lemma 2.1]{MoNg14}. This proves (i)$\Rightarrow$(iii). We now claim that (iii) implies that $C_L$ is positive semidefinite for every $L\in \Sp F(\xb,\yb)$. Assuming that $C_L$ is not positive semidefinite for some $L\in \Sp F(\xb,\yb)$, the same arguments as in the proof of Theorem \ref{ThStrRegHypo} can be used to obtain  \eqref{EqAuxPseudoMon} contradicting the local monotonicity of $F$. Hence our claim holds true and the implication (iii)$\Rightarrow$(i) follows from Theorem \ref{ThStrRegHypo}.
\end{proof}

\begin{remark}\label{RemPosDefCoder}Theorem \ref{ThStrRegHypo} improves the sufficient conditions for strong metric regularity obtained by Nghia et al \cite{NgPhTr20}. E.g., in \cite[Corollary 3.11]{NgPhTr20} it is shown that $F:\R^n\tto\R^n$ is strongly metrically regular around $(\xb,\yb)$ if
\begin{enumerate}
\item[(a)]$F$ is locally hypomonotone at $(\xb,\yb)$ and
\item[(b)] $D^*F(\xb,\yb)$ is positive definite in the sense that
$\skalp{u^*,v^*}>0$ holds for all $u^*\in D^*F(\xb,\yb)(v^*)$, $v^*\not=0$.
\end{enumerate}
We now show that these assumptions imply the assumptions of Theorem \ref{ThStrRegHypo}(i).
Indeed, by the positive definiteness of the coderivative $D^*F(\xb,\yb)$  together with the Mordukhovich criterion \eqref{EqMoCrit} we may deduce that $F$ is metrically regular around $(\xb,\yb)$  and therefore \SCD regular as well. Further, for every $L\in\Sp F(\xb,\yb)$ we have $\rge(C_L^T,I)\subseteq \gph D^*F(\xb,\yb)$ by Proposition \ref{PropC_L} and Remark \ref{RemSymmetry} implying $\skalp{C_L^Tp,p}>0$ for all $p$ with $C_L^Tp\not=0$ by assumption (b). Hence, $C_L$ is positive semidefinite. By assumption (a) there is some $\gamma\geq 0$ such that $\gamma I+F$ is locally monotone at $(\xb,\yb)$ and, since $D^*(\gamma I+F)(\xb,\gamma\xb+\yb)=\gamma I+D^*F(\xb,\yb)$ is positive definite, we conclude from the Mordukhovich criterion that $\gamma I+ F$ is metrically regular around $(\xb,\gamma\xb+\yb)$. Thus, by Corollary \ref{CorStrRegMonMap}, the mapping $\gamma I+F$ is locally maximally monotone at $(\xb,\gamma \xb+\yb)$ and  $F$ is locally maximally hypomonotone at $(\xb,\yb)$ by the definition.
Hence, we have shown that the assumptions of Theorem \ref{ThStrRegHypo}(i) are weaker than those of \cite[Corollary 3.11]{NgPhTr20}. When we now consider, e.g., the mapping $F(x_1,x_2)=(-x_2,x_1)$ and an arbitrary reference point $(\xb,F(\xb))$, we observe that the positive definiteness assumption (b) is not fulfilled. Nevertheless, Theorem \ref{ThStrRegHypo} works well and so it in fact improves the mentioned statement in \cite{NgPhTr20}.
\end{remark}

Next we  turn our attention to the characterization of tilt-stable minimizers by \SCD regularity of the subdifferential.

\begin{definition} [\bf tilt-stable minimizers]\label{DefTiltStab} Let $q:\R^n\to\bar \R$, and let $\xb\in\dom q$. Then:
\begin{enumerate}
\item[(i)] $\xb$ is a {\em tilt-stable local minimizer} of $q$ if there is a number $\gamma>0$ such that the mapping
\begin{equation}\label{EqM_gamma}
M_\gamma(x^\ast):=\argmin\big\{q(x)-\skalp{x^\ast,x}\mv x\in\B_\gamma(\xb)\big\},\quad x^*\in\R^n,
\end{equation}
is single-valued and Lipschitz continuous in some neighborhood of $\bar x^*=0\in\R^n$ with $M_\gamma(0)=\{\xb\}$.
\item[(ii)] The {\em exact bound of tilt stability} of $q$ at $\xb$ is defined by
\begin{eqnarray}\label{tilt-exact}
{\rm tilt\;}(q,\xb):=\limsup_{\AT{v^*,w^*\to0}{v^*\not=w^*}}\frac{\norm{M_\gamma(v^*)-M_\gamma(w^*)}}{\norm{v^*-w^*}}.
\end{eqnarray}
\end{enumerate}
\end{definition}
The theory developed in Section \ref{SecSemiSm} enables us to provide a new characterization of tilt-stable local minimizers.

\begin{theorem}\label{ThTiltStab}
  For a function $q:\R^n\to\bar \R$ having $0\in\partial q(\xb)$ and such that $q$ is both prox-regular and subdifferentially continuous at  $\xb$ for $\xba=0$, the following statements are equivalent:
  \begin{enumerate}
  \item[(i)] $\xb$ is a tilt-stable local minimizer of $q$.
  \item[(ii)] $\partial q$ is \SCD regular around $(\xb,0)$ and $C_L$ is positive semidefinite for every $L\in\Sp \partial q(\xb,0)$.
  \end{enumerate}
  Further, if $\xb$ is a tilt-stable minimizer of $q$ then ${\rm tilt\;}(q,\xb)=\scdreg\partial q(\xb,0)$.
\end{theorem}
\begin{proof}
 If $\xb$ is a tilt-stable minimizer, we may conclude from \cite[Theorem 1.3]{PolRo98} that the mapping $M_\gamma$ is a single-valued Lipschitzian localization of $\partial q^{-1}$ around $(0,\xb)$ so that $\partial q$ is strongly metrically regular around $(\xb,0)$ and ${\rm tilt\;}(q,\xb)=\reg \partial q(\xb,0)=\scdreg \partial q(\xb,0)$.
 By \cite[Theorem 1.3]{PolRo98}, statement (i) is equivalent to the condition
 \begin{enumerate}
   \item[(iii)] The coderivative $D^*\partial q(\xb,0)$ is positive definite in the sense that
   \begin{equation}\label{EqPosDefCoder}
     \skalp{v^*,u^*}>0\quad\mbox{whenever}\quad u^*\in D^*F(\xb,\yb)(v^*),\ v^*\not=0.
   \end{equation}
 \end{enumerate}
So it suffices to prove the equivalence (ii)$\Leftrightarrow$(iii).\par
 {\em Proof that (ii) $\Rightarrow$ (iii)}. By Proposition \ref{PropProxRegularQ}, there is some $\lambda >0$ such that $\partial q$ is graphically Lipschitzian with transformation mapping $\Phi(x,x^*)=(x+\lambda x^*,x)$. Further, any $Z\in\overline{\nabla}{}^\Phi\partial q(\xb,0)$ satisfies $\rge Z=(\rge Z)^*$ and is of the form $Z=\myvec{B\\\frac 1\lambda(I-B)}$, where $B$ is some symmetric positive semidefinite $n\times n$ matrix. Consider $Z\in \overline{\nabla}{}^\Phi\partial q(\xb,0)$ and set $B:=\pi_1(Z)$. Since $\partial q$ is \SCD regular around $(\xb,0)$, the matrix $\pi_2(Z)=\frac 1\lambda (I-B)$ is nonsingular by Proposition \ref{PropC_L} and $C_{\rge Z}=\lambda B(I-B)^{-1}$. If the eigenvalues of $B$ are denoted by $\mu_1,\ldots,\mu_n$, the eigenvalues of $C_L$ are $\lambda\mu_i/(1-\mu_i)$, $i=1,\ldots,n$, and, together with $\mu_i\geq 0$, we conclude that $C_L$ is positive semidefinite if and only if $\max \mu_i =\norm{B}<1$. Since $\overline{\nabla}{}^\Phi \partial q(\xb;0)$ is compact, it follows that $\eta:=\max\{\norm{\pi_1(Z)}\mv Z\in \overline{\nabla}{}^\Phi \partial q(\xb;0)\}<1$. Now consider $(v^*,u^*)\in \gph D^*F(\xb,\yb)$. By Proposition \ref{PropInclStrDer} there is some $\bar Z\in\co\overline{\nabla}{}^\Phi\partial q(\xb,0)$ such that $(v^*,u^*)\in(\rge \bar Z)^*$. The matrix $\bar B:=\pi_1(\bar Z)$ is a convex combination of symmetric positive semidefinite matrices $B_i$ satisfying $\norm{B_i}\leq \eta$. Thus $\bar B$ is  symmetric positive semidefinite and $\norm{\bar B}\leq \eta <1$. By taking into account $\pi_2(\bar Z)=\frac 1\lambda(I-\bar B)$,  $(\rge \bar Z)^\perp=\rge(\frac 1\lambda(I-\bar B,-\bar B)$ and $(\rge \bar Z)^*=\rge\bar Z$ follows. Thus we may find some $p\in\R^n$ with $(v^*,u^*)=(\bar Bp, \frac 1\lambda(I-\bar B)p$ to obtain
 \[\skalp{v^*,u^*}=\frac 1\lambda p^T\bar B(I-\bar B)p.\]
 The matrix $\bar B(I-\bar B)$ is symmetric and has eigenvalues $\mu_i(1-\mu_i)$, where $\mu_1,\ldots,\mu_n$ are the eigenvalues of $\bar B$. Since $0\leq\mu_i\leq\eta<1$, $i=1,\ldots,n$, the matrix $\bar B(I-\bar B)$ is positive semidefinite implying $\skalp{v^*,u^*}\geq 0$. Further, $p^T\bar B(I-\bar B)p$ vanishes if and only if $p$ is a linear combination of eigenvectors associated with the zero eigenvalues of $\bar B$, i.e., $\bar Bp=v^*=0$ and \eqref{EqPosDefCoder} follows .\par
{\em Proof that (iii) $\Rightarrow$ (ii)}.
Condition \eqref{EqPosDefCoder} implies that the Mordukhovich criterion \eqref{EqMoCrit} is fulfilled and we may conclude that $\partial q$ is \SCD regular around $(\xb,0)$. Further, for every $L\in\Sp F(\xb,0)$ we have $L^*=\rge(C_L^T,I)\subseteq \gph D^*\partial q(\xb,0)$ by Proposition \ref{PropC_L} and Lemma \ref{LemSCDname} and therefore $(C_L^Tp,p)\in \gph D^*\partial q(\xb,0)$ $\forall p$. From \eqref{EqPosDefCoder} we deduce $\skalp{C_L^Tp,p}>0$ for all $p$ with $C_L^Tp\not=0$ and the positive semidefiniteness of $C_L^T$ and $C_L$ follows.
\end{proof}
\begin{example}Consider again the mapping $F:\R^2\to\R^2$ from Example \ref{ExSubspReg}. If $h=\nabla \phi$ for some potential $\phi:\R^n\to\R$, we see  that the inclusion $0\in F(x)$ describes in fact the first-order optimality condition for the optimization problem
\[\min_{x\in C}\frac 12 x_1^2-\frac 12 x_2^2+\phi(x).\]
Since $C_{TL_1}$ is not positive semidefinite, we conclude from Theorem \ref{ThTiltStab} that $\xb=0$ is not a tilt stable local minimizer. This is also in accordance with \cite[Theorem 4.5]{PolRo98}.
\end{example}

\section{Conclusion}
   Subspaces contained in the graph of the limiting coderivative may definitely serve as a basis for
construction of suitable local approximations in the broad class of SCD multifunctions. It came to us,
however, as a surprise that these subspaces and their counterparts in case of the limiting outer graphical
derivative contain a lot of information about stability behavior of the considered mappings.
    The developed theory makes use of notions, mimicking the generalized derivatives and coderivatives
in the ''standard'' generalized differential calculus. However, their structure is, in most cases, somewhat simpler when
compared with the standard notions and so the derived new characterization of strong metric
(sub)regularity and tilt stability are typically easier to work with.
   Finally, let us point out that the property of strong metric subregularity around the reference point,
characterized via the subspaces contained in the limiting outer graphical derivative, seems to be a weak
stability property ensuring the locally superlinear convergence of the semismooth* Newton
method.

\section*{Acknowledgements}
The research of the first author was  supported by the Austrian Science Fund
(FWF) under grant P29190-N32. The  research of the second author was supported by the Grant Agency of the
Czech Republic, Project 21-06569K, and the Australian Research Council, Project DP160100854.
Further the authors would like to express their gratitude to an anonymous reviewer and to M.~Fabi\'an for careful reading and many important suggestions.

\end{document}